\documentclass[a4paper,12pt]{article}

\usepackage[left=2cm,right=2cm, top=2cm,bottom=3cm,bindingoffset=0cm]{geometry}

\usepackage{verbatim}
\usepackage{amsmath}
\usepackage{amsthm}
\usepackage{amssymb}
\usepackage{delarray}
\usepackage{cite}
\usepackage{hyperref}
\usepackage{mathrsfs}
\usepackage{tikz}
\usetikzlibrary{patterns}
\usepackage{caption}
\DeclareCaptionLabelSeparator{dot}{. }
\newcommand{\al}{\alpha}
\newcommand{\be}{\beta}
\newcommand{\ga}{\gamma}
\newcommand{\de}{\delta}
\newcommand{\la}{\lambda}
\newcommand{\om}{\omega}

\theoremstyle{plain}

\numberwithin{equation}{section}

\newtheorem{thm}{Theorem}[section]
\newtheorem{lem}[thm]{Lemma}
\newtheorem{prop}[thm]{Proposition}
\newtheorem{cor}[thm]{Corollary}

\theoremstyle{definition}

\newtheorem{alg}[thm]{Algorithm}
\newtheorem{ip}[thm]{Inverse Problem}

\theoremstyle{remark}

\newtheorem{remark}[thm]{Remark}

\DeclareMathOperator*{\Res}{Res}

 \allowdisplaybreaks

\begin{document}

\begin{center}
{\Large\bf McLaughlin's inverse problem \\[0.2cm] for the fourth-order differential operator}
\\[0.2cm]
{\bf Natalia P. Bondarenko} \\[0.2cm]
\end{center}

\vspace{0.5cm}

{\bf Abstract.} In this paper, we revisit McLaughlin's inverse problem, which consists in the recovery of the fourth-order differential operator from the eigenvalues and two sequences of weight numbers. We for the first time prove the uniqueness for solution of this problem. Moreover, we obtain the interpretation of McLaughlin's problem in the framework of the general inverse problem theory by Yurko for differential operators of arbitrary orders. An advantage of our approach is that it requires neither smoothness of the coefficients nor self-adjointness of the operator. In addition, we establish the connection between McLaughlin's problem and Barcilon's three-spectra inverse problem.

\medskip


\medskip


\vspace{1cm}

\section{Introduction} \label{sec:intr}

In this paper, we consider the following boundary value problem $\mathcal L = \mathcal L(p, q, a, b, c)$:
\begin{gather} \label{eqv}
\ell(y) := y^{(4)} - (p(x) y')' + q(x) y = \la y, \quad x \in (0,1), \\ \label{bc}
\left.\begin{array}{c}
U_1(y) := y''(0) + ay'(0) - b y(0) = 0, \\
U_2(y) := y^{[3]}(0) + b y'(0) + c y(0) = 0, \\
y(1) = y'(1) = 0,
\end{array}\qquad \right\}
\end{gather}
where $p$ and $q$ are complex-valued functions of $L_1[0,1]$, $\la$ is the spectral parameter, $a, b, c$ are complex constants, $y^{[3]} := y''' - p y'$ is the quasi-derivative, and a function $y$ belongs to the domain
$$
\mathcal D := \{ y \in W_1^3[0,1] \colon y^{[3]} \in W_1^1[0,1] \}.
$$

Suppose that the eigenvalues $\{ \la_n \}_{n \ge 1}$ of the problem $\mathcal L$ are simple and denote by $\{ y_n(x) \}_{n \ge 1}$ the corresponding eigenfunctions normalized by the condition $\int_0^1 y_n^2(x) \, dx = 1$. Put $\ga_n := y_n(0)$, $\xi_n := y_n'(0)$. This paper is concerned with the following inverse spectral problem.

\begin{ip} \label{ip:main}
Given the spectral data $\{ \la_n, \ga_n, \xi_n \}_{n \ge 1}$, find the coefficients $p, q, a, b, c$ of the problem $\mathcal L$.
\end{ip}

Inverse Problem~\ref{ip:main} has been introduced by McLaughlin \cite{McL82} as an attempt to transfer the classical results regarding the second-order inverse Sturm-Lioville problems to the order four. Indeed, it is well known that a real-valued potential $q(x)$ of the Sturm-Liouville problem 
\begin{equation} \label{StL}
-y'' + q(x) y = \la y, \quad x \in (0,1), \quad y(0) = y(1) = 0
\end{equation}
is uniquely specified by the eigenvalues $\{ \la_n \}_{n \ge 1}$ of \eqref{StL} and the weight numbers $\al_n := y_n'(0)$, $n \ge 1$, where $\{ y_n(x) \}_{n \ge 1}$ are the eigenfunctions normalized by the condition $\int_0^1 y_n^2(x) \, dx = 1$ (see, e.g., \cite{Mar77, Lev84, PT87, FY01}). The most complete results for the inverse Sturm-Liouville inverse problems have been obtained by the method of Gelfand and Levitan \cite{GL51} based on transformation operators. However, that method turned out to be ineffective for higher-order differential operators.

McLaughlin has considered the problem \eqref{eqv}--\eqref{bc} under the assumptions $p \in C^3[0,1]$, $q \in C^1[0,1]$ in \cite{McL82, McL86} and $p \in C^1[0,1]$, $q \in C[0,1]$ in \cite{McL84}. Furthermore, she assumed that all the coefficients $p(x)$, $q(x)$, $a$, $b$, and $c$ were real, so the problem $\mathcal L$ was self-adjoint. McLaughlin has proved some solvability theorems for Inverse Problem~\ref{ip:main} under the condition of existence of the transformation operator. This condition is very restrictive. The existence of transformation operators for higher-order differential equations and applications to inverse problem theory have been intensively investigated by Sakhnovich \cite{Sakh61, Sakh62}, Khachatryan \cite{Khach83}, Malamud \cite{Mal82, Mal90} and other mathematicians. Their results show that transformation operators exist only under some requirements of analyticity for differential equation coefficients. However, higher-order inverse spectral problems with piecewise-analytical coefficients can be effectively solved by the method of standard models (see \cite{Yur89, Yur00}). Anyway, although the studies of McLaughlin \cite{McL82, McL84, McL86} did not imply fundamental theoretical results for Inverse Problem~\ref{ip:main}, they were useful for the development of numerical methods (see, e.g., \cite{Glad86, Glad05}). A variety of inverse spectral problems for the fourth-order differential operators in other statements were considered in \cite{Bar74-1, Bar74-2, Iw88, PK97, CPS98, Pap04, Mor15, BK15, AZ19, PB20, JLX22}.

Fourth-order linear eigenvalue problems arise in various applications. First of all, the transverse vibrations of a beam are described by the Euler-Bernoulli equation (see \cite{Glad05}):
$$
(A(x) u''(x))'' = \la B(x) u(x),
$$
which can be transformed to equation \eqref{eqv}. Barcilon \cite{Bar74-1} and McLaughlin \cite{McL86} investigated the fourth-order inverse spectral problems in connection with geophysics. Furthermore, fourth-order linear differential operators arise in mechanics, optics, and acoustics (see \cite{Mikh64, YS72, Pol20}).

In this paper, we revisit McLaughlin's problem and consider it in a more general case: with non-smooth coefficients $p, q \in L_1[0,1]$ and without the requirement of the self-adjointness. Our main goal is to prove the uniqueness for the inverse problem solution. Note that McLaughlin \cite{McL82, McL84, McL86} did not study this issue, since she considered the problem in a very special case of existence of the transformation operator. Thus, to the best of the author's knowledge, the question of uniqueness for the recovery of the problem $\mathcal L$ from the spectral data $\{ \la_n, \ga_n, \xi_n \}_{n \ge 1}$ was open. In the present paper, the author not only aims to fill this gap but also to interpret McLaughlin's problem in the framework of the general inverse problem theory that has been created by Yurko \cite{Yur00, Yur02, Yur92, Yur93, Yur95} for higher-order differential equations 
\begin{equation} \label{ho}
y^{(n)} + \sum_{k = 0}^{n-2} p_k(x) y^{(k)} = \la y, \quad n \ge 2.
\end{equation}

Yurko has found such spectral data that uniquely specify the coefficients $\{ p_k(x) \}_{k = 0}^{n-2}$ of equation \eqref{ho} for any integer order $n \ge 2$ on a finite interval and on the half-line independently of the behavior of the spectra (see \cite{Yur92, Yur02}). These spectral data are an $(n \times n)$ meromorphic matrix function $M(\la) = [m_{jk}(\la)]_{j,k = 1}^n$, which is called \textit{the Weyl-Yurko matrix}. Moreover, Yurko has developed \textit{the method of spectral mappings}, which allowed him not only to solve constructively the inverse spectral problems for equation \eqref{ho} but also to obtain the necessary and sufficient conditions for their solvability. In the case when the poles of the Weyl-Yurko matrix are simple, it is convenient to reconstruct the coefficients $\{ p_k \}_{k = 0}^{n-2}$ on a finite interval from the poles $\Lambda$ and the weight matrices $\{ \mathcal N(\la_0) \}_{\la_0 \in \Lambda}$, which generalize the weight numbers $\{ \al_n \}_{n \ge 1}$ for the Sturm-Liouville equation \eqref{StL}.
However, the studies of Yurko are limited to \eqref{ho} either with sufficiently smooth coefficients $p_k \in W_1^k[0,1]$, $k = \overline{0,n-2}$ (see \cite{Yur92}), or  with Bessel-type singularities (see \cite{Yur93, Yur95}). Rigorously speaking, the results of Yurko cannot be directly applied to equation \eqref{eqv}  with $p, q \in L_1[0,1]$.
Therefore, in this paper, we apply the recent results of Bondarenko \cite{Bond21, Bond22-alg, Bond23-mmas, Bond23-reg}. In those studies, the inverse problem theory has been transferred to higher-order differential operators with distribution coefficients of the Mirzoev-Shkalikov class (see \cite{MS16}). In particular, that class contains equation \eqref{eqv} with $p \in W_2^{-1}[0,1]$ and $q \in W_2^{-2}[0,1]$.

Thus, in this paper, we construct the Weyl-Yurko matrix $M(\la)$ and the discrete spectral data $\{ \la_0, \mathcal N(\la_0) \}_{\la_0 \in \Lambda}$ due to the studies of Bondarenko \cite{Bond21, Bond22-alg, Bond23-mmas, Bond23-reg}. Then, we establish relations between these universal spectral characteristics and McLaughlin's spectra data $\{ \la_n, \ga_n, \xi_n \}_{n \ge 1}$.
This implies the uniqueness for solution of Inverse Problem~\ref{ip:main}. In addition, we compare McLaughlin's problem with the inverse problem of Barcilon \cite{Bar74-1, Bar74-2}, which consists in the recovery of the fourth-order differential operator from three spectra.
Our investigation requires a comprehensive analysis of various spectral characteristics for equation \eqref{eqv} and consideration of several cases of their behavior. An important role in our analysis is played by some symmetries of the Weyl-Yurko matrix, which follow from the special formally self-adjoint structure of the problem $\mathcal L$. 
Since the method of spectral mappings is constructive, in the future, our results can be applied to the development of new reconstruction techniques and to obtaining the necessary and sufficient conditions for the solvability of McLaughlin's inverse problem.

The paper is organized as follows. In Section~\ref{sec:main}, we formulate the main results and explain the proof strategy. Section~\ref{sec:prelim} contains preliminaries and several auxiliary lemmas. In Section~\ref{sec:gen}, we prove the uniqueness of recovering the problem $\mathcal L$ from particular elements of the Weyl-Yurko matrix and from the discrete spectral data under some additional conditions. Sections~\ref{sec:uniq1} and~\ref{sec:uniq2} are devoted to the proofs of the uniqueness theorems for McLaughlin's Inverse Problem~\ref{ip:main}. In Section~\ref{sec:Bar}, we establish the connection between McLaughlin's and Barcilon's inverse problems.  

Throughout the paper, we use the following notations:
\begin{enumerate}
\item The prime $y'(x, \la)$ denotes the differentiation with respect to $x$ and the dot $\dot y(x, \la)$, with respect to $\la$.
\item $\de_{jk} = \begin{cases} 1, \quad j = k, \\ 0, \quad j \ne k\end{cases}$ is the Kronecker delta.
\item The notation $\{ A_{\langle k \rangle}(\la_0) \}$ is used for the coefficients of the Laurent series of a function $A(\la)$ at the point $\la_0$:
$$
A(\la) = \sum_{k = -\infty}^{\infty} A_{\langle k \rangle}(\la_0) (\la - \la_0).
$$
In particular, $A_{\langle -1 \rangle}(\la_0) = \Res_{\la = \la_0} A(\la)$.
\item In the proofs of the uniqueness theorems, along with the problem $\mathcal L = \mathcal L(p,q,a,b,c)$, we consider another problem $\tilde {\mathcal L} = \mathcal L(\tilde p, \tilde q, \tilde a, \tilde b, \tilde c)$ of the same form but with different coefficients. We agree that, if a symbol $\al$ denotes an object related to $\mathcal L$, then the symbol $\tilde \al$ denotes the analogous object related to $\tilde {\mathcal L}$.
\end{enumerate}

\section{Main results and proof strategy} \label{sec:main}

Along with $U_1$ and $U_2$ in \eqref{bc}, define the linear forms
\begin{equation} \label{defUV}
U_3(y) = y(0), \quad U_4(y) = y'(0), \quad V_s(y) = y^{[s-1]}(1), \quad s = \overline{1,4}, 
\end{equation}
where $y^{[j]} := y^{(j)}$ for $j = 0, 1, 2$ and $y^{[3]} := y''' - p y$.
Then, the boundary conditions \eqref{bc} can be rewritten as
\begin{equation} \label{bc2}
U_1(y) = U_2(y) = 0, \quad V_1(y) = V_2(y) = 0.
\end{equation}

Denote by $\{ C_k(x, \la) \}_{k = 1}^4$ and $\{ \Phi_k(x, \la) \}_{k = 1}^4$ the solutions of equation \eqref{eqv} satisfying the initial conditions
\begin{equation} \label{initC}
U_s(C_k) = \de_{sk}, \quad s = \overline{1,4}, 
\end{equation}
and the boundary conditions
\begin{equation} \label{bcPhi}
U_s(\Phi_k) = \de_{sk}, \quad s = \overline{1,k}, \qquad 
V_j(\Phi_k) = 0, \quad j = \overline{1,4-k},
\end{equation}
respectively. Clearly, for each fixed $x \in [0,1]$, the functions $C_k^{[j]}(x, \la)$ are entire in $\la$ and $\Phi_k^{[j]}(x, \la)$ are meromorphic in $\la$ for $k = \overline{1,4}$, $j = \overline{0,3}$. Furthermore, $\{ C_k(x, \la) \}_{k = 1}^4$ is a fundamental system of solutions of equation \eqref{eqv}, so the solutions $\{ \Phi_k(x, \la) \}_{k = 1}^4$ can be expanded over this system. Therefore, in terms of the matrix functions $C(x, \la) := [C_k^{[j-1]}(x, \la)]_{j,k = 1}^4$, $\Phi(x, \la) = [\Phi_k^{[j-1]}(x, \la)]_{j,k = 1}^4$, we have the relation
\begin{equation} \label{relM}
\Phi(x, \la) = C(x, \la) M(\la),
\end{equation}
where $M(\la) = [m_{jk}(\la)]_{j,k = 1}^4$ is a meromorphic matrix function, which is called \textit{the Weyl-Yurko matrix}. Theorem~5.3 in \cite{Bond23-mmas} implies the following proposition for the problem $\mathcal L$.

\begin{prop}[\cite{Bond23-mmas}] \label{prop:uniqM}
The Weyl-Yurko matrix $M(\la)$ uniquely specifies the coefficients $p, q, a, b, c$.
\end{prop}

Let us consider the properties of the Weyl-Yurko matrix. Using \eqref{initC}, \eqref{bcPhi}, and \eqref{relM}, one can easily show that $M(\la)$ is a unit lower-triangular matrix:
\begin{equation} \label{matrM}
M(\la) = \begin{bmatrix}
            1 & 0 & 0 & 0 \\
            m_{21}(\la) & 1 & 0 &  0 \\
            m_{31}(\la) & m_{32}(\la) & 1 & 0 \\
            m_{41}(\la) & m_{42}(\la) & m_{43}(\la) & 1
        \end{bmatrix}.
\end{equation}
Moreover, its elements can be represented in the form
\begin{equation} \label{mjk}
m_{jk}(\lambda)=-\frac{\Delta_{jk}(\lambda)}{\Delta_{kk}(\lambda)}, \quad 1\leq k \;\textless j \leq 4,
\end{equation}
where $\Delta_{kk}(\lambda):=\det[V_{5-s}(C_{r})]_{s,r=k+1}^{4}$ and $\Delta_{jk}(\lambda)$ is obtained from $\Delta_{kk}(\lambda)$ by replacing $C_{j}$ by $C_{k}$. Clearly, the functions $\Delta_{jk}(\la)$, $1 \le k \le j \le 4$, are entire in $\la$, so $m_{jk}(\la)$ are meromorphic in $\la$. On the other hand, the zeros of the functions $\Delta_{jk}(\la)$ coincide with the eigenvalues of the corresponding boundary value problems $\mathcal L_{jk}$ for equation \eqref{eqv} with the boundary conditions
\begin{equation} \label{bcjk}
U_{\xi}(y)=0, \quad \xi=\overline{1,k-1},j, \quad V_{\eta}(y)=0, \quad \eta=\overline{1, 4-k}.
\end{equation}
In particular, $\mathcal L_{22} = \mathcal L$, so the zeros of $\Delta_{22}(\la)$ are $\{ \la_n \}_{n \ge 1}$.

Introduce the so-called \textit{separation condition}:

\smallskip

$(\mathcal S)$: For $k = 1, 2$, the functions $\Delta_{kk}(\la)$ and $\Delta_{k+1,k+1}(\la)$ do not have common zeros.

\smallskip

It will be shown in Lemma~\ref{lem:aux1} that $\Delta_{11}(\la) \equiv -\Delta_{33}(\la)$. Thus, it is sufficient to replace $k = 1, 2$ by $k = 1$ in $(\mathcal S)$. Under the separation condition, a less amount of information than the whole Weyl-Yurko matrix can be used to uniquely determine the problem $\mathcal L$. 

\begin{thm} \label{thm:Leib}
Under the condition $(\mathcal S)$, the functions $m_{21}(\la)$, $m_{32}(\la)$, and $m_{43}(\la)$ uniquely specify the coefficients $p,q,a,b,c$ of the problem $\mathcal L$.
\end{thm}

Analogs of Theorem~\ref{thm:Leib} for equation \eqref{ho} with sufficiently smooth coefficients were proved by Leibenzon \cite{Leib66} in the case of simple eigenvalues and by Yurko \cite{Yur02} in the general case. For the problem $\mathcal L$, the proof of Theorem~\ref{thm:Leib} is provided in Section~\ref{sec:gen}.

Relying on Theorem~\ref{thm:Leib}, we prove the following uniqueness theorem for the solution of Inverse Problem~\ref{ip:main}.

\begin{thm} \label{thm:uniq1}
Suppose that $p$ and $q$ are complex-valued functions of $L_1[0,1]$, $a, b, c \in \mathbb C$, the eigenvalues $\{ \la_n \}_{n \ge 1}$ of the problem $\mathcal L$ are simple, and the condition $(\mathcal S)$ holds. Then, the spectral data $\{ \la_n, \xi_n, \ga_n \}_{n \ge 1}$ uniquely specify $p, q, a, b, c$.
\end{thm}

The condition $(\mathcal S)$ significantly simplifies the analysis of the inverse problem, however, it is not necessary. The following theorem establishes the uniqueness without the separation condition.

\begin{thm} \label{thm:uniq2}
Suppose that $p$ and $q$ are complex-valued functions of $L_1[0,1]$, $a, b, c \in \mathbb C$, and the entire functions $\Delta_{11}(\la)$ and $\Delta_{22}(\la)$ have only simple zeros. Then, the spectral data $\{ \la_n, \xi_n, \ga_n \}_{n \ge 1}$ uniquely specify $p, q, a, b, c$.
\end{thm}

In view of \eqref{mjk}, under the conditions of Theorem~\ref{thm:uniq2}, the Weyl-Yurko matrix has only simple poles. Then, the Laurent series with respect to any pole $\la_0$ has the form
$$
M(\la) = \frac{M_{\langle -1 \rangle}(\la_0)}{\la - \la_0} + M_{\langle 0 \rangle}(\la_0) + M_{\langle 1\rangle}(\la_0) (\la - \la_0) + \dots,
$$
where $M_{\langle j \rangle}$ are $(4 \times 4)$ constant matrices. Denote
\begin{equation} \label{defN}
\mathcal N(\la_0) := (M_{\langle 0 \rangle}(\la_0))^{-1} M_{\langle -1 \rangle}(\la_0).
\end{equation}

Then, instead of the Weyl-Yurko matrix $M(\la)$, one can consider the discrete spectral data $\{ \la_0, \mathcal N(\la_0) \}_{\la_0 \in \Lambda}$, where $\Lambda$ is the set of the poles of $M(\la)$. The uniqueness of recovering the higher-order differential operators from the discrete spectral data has been proved in \cite{Yur02} for regular coefficients and in \cite{Bond23-reg} for distribution coefficients. However, in \cite{Yur02} and \cite{Bond23-reg}, other types of boundary conditions were considered, so we prove the following theorem:

\begin{thm} \label{thm:sd}
Suppose that the functions $\Delta_{11}(\la)$ and $\Delta_{22}(\la)$ have only simple zeros. Then the spectral data $\{ \la_0, \mathcal N(\la_0) \}_{\la_0 \in \Lambda}$ uniquely specify $p,q,a,b,c$.
\end{thm}

In order to prove Theorem~\ref{thm:uniq2}, we will show that the spectral data $\{ \la_n, \xi_n, \ga_n \}_{n \ge 1}$ uniquely determine $\{ \la_0, \mathcal N(\la_0) \}_{\la_0 \in \Lambda}$. For this purpose, we need to study the structure of the weight matrices $\mathcal N(\la_0)$, which differs in the following five cases:
\begin{equation} \label{cases}
\left.\begin{array}{rl}
    (I) \colon & \quad U_3(y_n) \ne 0, \quad V_3(y_n) \ne 0, \\
    (II) \colon & \quad U_3(y_n) \ne 0, \quad V_3(y_n) = 0, \\
    (III) \colon & \quad U_3(y_n) = 0, \quad V_3(y_n) \ne 0, \\
    (IV) \colon & \quad U_3(y_n) = 0, \quad V_3(y_n) = 0, \\
    (V) \colon & \quad \la_0 \not\in \{ \la_n \}_{n \ge 1}.
\end{array} \qquad \right\}
\end{equation}
(In the cases (I)--(IV), we mean that $\la_0 = \la_n$). If the zeros of $\Delta_{11}(\la)$ can be multiple and the separation condition $(\mathcal S)$ does not hold, then the structure of the Laurent series for $M(\la)$ becomes more complicated and the inverse problem, very technically difficult, so we exclude this case from consideration. Anyway, Theorems~\ref{thm:uniq1} and~\ref{thm:uniq2} show that the uniqueness for solution of McLaughlin's problem does not require:
\begin{itemize}
\item self-adjointness;
\item smoothness of $p$ and $q$;
\item separation condition.
\end{itemize}

\section{Preliminaries} \label{sec:prelim}

In this section, we study the properties of the Weyl-Yurko matrix, obtain useful relations and asymptotic estimates for the characteristic functions $\Delta_{jk}(\la)$, and provide other preliminaries.

In the standard way, equation \eqref{eqv} for $y \in \mathcal D$ can be represented as the first-order system
\begin{equation} \label{sys}
Y' = (F(x) + \Lambda) Y, \quad x \in (0,1),
\end{equation}
where
\begin{equation} \label{defF}
Y := \begin{bmatrix}
 y \\ y' \\ y'' \\ y^{[3]}
\end{bmatrix}, \quad
F(x) := \begin{bmatrix}
            0 & 1 & 0 & 0 \\
            0 & 0 & 1 & 0 \\
            0 & p(x) & 0 & 1\\
            -q(x) & 0 & 0 & 0
        \end{bmatrix}, \quad
\Lambda := \begin{bmatrix}
                0 & 0 & 0 & 0 \\
                0 & 0 & 0 & 0 \\
                0 & 0 & 0 & 0 \\
                \la & 0 & 0 & 0             
            \end{bmatrix}.
\end{equation}

The linear forms $\{ U_k \}_{k =1}^4$ and $\{ V_k \}_{k = 1}^4$ can be represented by the matrices
\begin{equation} \label{matrUV}
U = \begin{bmatrix}
        -b & a & 1 & 0 \\
        c & b & 0 & 1 \\
        1 & 0 & 0 & 0 \\
        0 & 1 & 0 & 0
    \end{bmatrix}, \quad V = \begin{bmatrix}
        1 & 0 & 0 & 0 \\
        0 & 1 & 0 & 0 \\
        0 & 0 & 1 & 0 \\
        0 & 0 & 0 & 1
    \end{bmatrix}
\end{equation}
The $k$-th rows of the matrices $U$ and $V$ contain the corresponding coefficients of the linear forms $U_k$ and $V_k$, respectively.

Denote by $\{ p_{k,0} \}_{k = 1}^4$ and $\{ p_{k,1} \}_{k = 1}^4$ the orders of the higher derivatives in the linear forms $\{ U_k \}_{k = 1}^4$ and $\{ V_k \}_{k = 1}^4$, respectively. According to \eqref{bc} and \eqref{defUV}, we have 
\begin{equation} \label{defp} 
p_{1,0} = 2, \quad p_{2,0} = 3, \quad p_{3,0} = 0, \quad p_{4,0} = 1, \quad p_{s,1} = s-1, \quad s = \overline{1,4}.
\end{equation}

\begin{lem} \label{lem:invM}
The following relation hold:
\begin{equation} \label{invM}
M^{-1}(\la) = \begin{bmatrix}
                   1 & 0 & 0 & 0 \\
                   -m_{43}(\la) & 1 & 0 & 0 \\
                   m_{42}(\la) & -m_{32}(\la) & 1 & 0 \\
                   -m_{41}(\la) & m_{31}(\la) & -m_{21}(\la) & 1
              \end{bmatrix}.
\end{equation}
Consequently, $m_{21}(\la) = m_{43}(\la)$ and
\begin{equation} \label{relm31}
m_{31}(\la) - m_{21}(\la) m_{32}(\la) + m_{42}(\la) = 0.
\end{equation}
\end{lem}

\begin{proof}
Following the approach of \cite[Section 2]{Bond22-alg}, we consider the matrices $F^{\star}(x) = [f_{k,j}^{\star}(x)]_{k,j = 1}^4$, $U^{\star}$, and $V^{\star}$ generated by the relations
$$
f_{k,j}^{\star}(x) := (-1)^{k + j + 1} f_{5 - j, 5 - k}^{\star}(x), \quad
U := (J_0^{-1} U^{-1} J_0)^T, \quad V := (J_1^{-1} V^{-1} J_1)^T,
$$
where $T$ denotes the matrix transpose, $J_s := [(-1)^{p_{k,s}^{\star}} \delta_{k,5-j}]_{k,j = 1}^4$, $p_{k,s}^{\star} := 3 - p_{5 - k,s}$, $s = 0, 1$. In our case, in view of \eqref{defF}, \eqref{defUV}, and \eqref{defp},
we have
\begin{equation} \label{defJ}
J_0 = J_1 = \begin{bmatrix}
                0 & 0 & 0 & 1 \\
                0 & 0 & -1 & 0 \\
                0 & 1 & 0 & 0 \\
                -1 & 0 & 0 & 0
            \end{bmatrix}
\end{equation}
and $F^{\star}(x) = F(x)$, $U^{\star} = U$, $V^{\star} = V$.

In \cite{Bond22-alg}, the following relation has been obtained for the Weyl-Yurko matrices $M(\la)$ and $M^{\star}(\la)$ constructed for the problems $(F(x), U, V)$ and $(F^{\star}(x), U^{\star}, V^{\star})$:
\begin{equation} \label{relMMs}
(M^{\star}(\la))^T J_0 M(\la) = J_0. 
\end{equation}
In our case, $M(\la) = M^{\star}(\la)$, so the relation \eqref{relMMs} implies
$$
M^{-1}(\la) = J_0^{-1} M^T(\la) J_0.
$$
Using \eqref{matrM} and \eqref{defJ}, we arrive at \eqref{invM}.
\end{proof}

For $y, z \in \mathcal D$, define the Lagrange bracket
$$
\langle y, z \rangle = y^{[3]} z - y''z' + y'z'' - y z^{[3]}.
$$
Using \eqref{bc} and \eqref{defUV}, we obtain the relation
\begin{equation} \label{brack0}
\langle y, z \rangle_{|x = 0} = U_2(y) U_3(z) - U_1(y) U_4(z) + U_4(y) U_1(z) - U_3(y) U_2(z).
\end{equation}

The following Lagrange identity can be proved by direct calculations:
$$
\int_0^1 \ell(y) z \, dx = \langle y, z \rangle\big|_0^1 + \int_0^1 y \ell(z) \, dx.
$$
In particular, if $\ell(y) = \la y$ and $\ell(z) = \mu z$, then
\begin{equation} \label{wron}
\langle y, z \rangle\big|_0^1 = (\la - \mu) \int_0^1 y z \, dx.
\end{equation}

For convenience, let us explicitly write down the determinants $\Delta_{jk}(\la)$, which appear in \eqref{mjk}:
\begin{gather} \label{D1}
\Delta_{11}(\la) = \begin{vmatrix}
  C_2''(1,\la) & C_3''(1,\la) & C_4''(1,\la) \\
  C_2'(1,\la) & C_3'(1,\la) & C_4'(1,\la) \\
  C_2(1,\la) & C_3(1,\la) & C_4(1,\la)
\end{vmatrix}, \quad
\Delta_{21}(\la) = \begin{vmatrix}
  C_1''(1,\la) & C_3''(1,\la) & C_4''(1,\la) \\
  C_1'(1,\la) & C_3'(1,\la) & C_4'(1,\la) \\
  C_1(1,\la) & C_3(1,\la) & C_4(1,\la)
 \end{vmatrix}, \\ \label{D2}
\Delta_{31}(\la) = \begin{vmatrix}
  C_2''(1,\la) & C_1''(1,\la) & C_4''(1,\la) \\
  C_2'(1,\la) & C_1'(1,\la) & C_4'(1,\la) \\
  C_2(1,\la) & C_1(1,\la) & C_4(1,\la)
\end{vmatrix}, \quad
\Delta_{41}(\la) =  \begin{vmatrix}
  C_2''(1,\la) & C_3''(1,\la) & C_1''(1,\la) \\
  C_2'(1,\la) & C_3'(1,\la) & C_1'(1,\la) \\
  C_2(1,\la) & C_3(1,\la) & C_1(1,\la)
\end{vmatrix}, \\ \label{D3}
\Delta_{22}(\la) = \begin{vmatrix}
C_3'(1,\la) & C_4'(1, \la) \\
C_3(1,\la) & C_4(1,\la)
\end{vmatrix}, \quad
\Delta_{32}(\la) = \begin{vmatrix}
C_2'(1,\la) & C_4'(1, \la) \\
C_2(1,\la) & C_4(1,\la)
\end{vmatrix}, \\
\Delta_{42}(\la) = 
\begin{vmatrix}
C_3'(1,\la) & C_2'(1, \la) \\ \label{D4}
C_3(1,\la) & C_2(1,\la)
\end{vmatrix}, \quad \Delta_{33}(\la) = C_4(1,\la), \quad \Delta_{43}(\la) = C_3(1,\la).
\end{gather}

Denote by $\{ S_k(x, \la) \}_{k = 1}^4$ the solutions of equation \eqref{eqv} satisfying the initial conditions
\begin{equation} \label{initS}
V_s(S_k) = \de_{sk}, \quad s = \overline{1,4}.
\end{equation}

\begin{lem} \label{lem:aux1}
The following relations hold:
\begin{align} \label{D11}
& \Delta_{11}(\la) = U_1(S_4) = -C_4(1, \la) = -\Delta_{33}(\la), \\ \label{D21}
& \Delta_{21}(\la) = -U_2(S_4) = -C_3(1,\la) = -\Delta_{43}(\la), \\ \label{D31}
& \Delta_{31}(\la) = -S_4(0, \la), \quad \Delta_{41}(\la) = -S_4'(0, \la).
\end{align}
\end{lem}

\begin{proof}
Consider the determinant
$$
\Delta(x, \la) := \begin{vmatrix}
                     C_2^{[3]}(x, \la) & C_3^{[3]}(x, \la) & C_4^{[3]}(x, \la) & S_4^{[3]}(x, \la) \\
                     C_2''(x, \la) & C_3''(x, \la) & C_4''(x, \la) & S_4''(x, \la) \\
                     C_2'(x, \la) & C_3'(x, \la) & C_4'(x, \la) & S_4'(x, \la) \\
                     C_2(x, \la) & C_3(x, \la) & C_4(x, \la) & S_4(x, \la)       
                  \end{vmatrix}.
$$
According to the Liouville-Ostrogradski formula, the Wronskian $\Delta(x, \la)$ of solutions of equation \eqref{eqv} does not depend on $x$. Using \eqref{initS} and \eqref{D1}, we easily get $\Delta(1, \la) = -\Delta_{11}(\la)$. On the other hand, taking linear combinations of the rows and applying the initial conditions \eqref{initC}, we obtain
$$
\Delta(0, \la) = \begin{vmatrix}
                     U_2(C_2) & U_2(C_3) & U_2(C_4) & U_2(S_4) \\
                     U_1(C_2) & U_1(C_3) & U_1(C_4) & U_1(S_4) \\
                     U_4(C_2) & U_4(C_3) & U_4(C_4) & U_4(S_4) \\
                     U_3(C_2) & U_3(C_3) & U_3(C_4) & U_3(S_4)
                  \end{vmatrix} = -U_1(S_4).
$$
Hence $\Delta_{11}(\la) = U_1(S_4)$. Next, consider the Lagrange bracket
$$
\langle S_4, C_4 \rangle = S_4^{[3]} C_4 - S_4'' C_4' + S_4' C_4'' - S_4 C_4^{[3]}.
$$
It follows from \eqref{wron} that 
$$
\langle S_4(x, \la), C_4(x, \la) \rangle_{|x = 0} = \langle S_4(x, \la), C_4(x, \la) \rangle_{|x = 1}.
$$
Using \eqref{bc}, \eqref{initC}, and \eqref{brack0} we get
\begin{align*}
\langle S_4(x, \la), C_4(x, \la) \rangle_{|x = 0} = & \, U_2(S_4) U_3(C_4) - U_1(S_4) U_4(C_4) \\ & + U_4(S_4) U_1(C_4) - U_3(S_4) U_2(C_4) = -U_1(S_4).
\end{align*}
Using \eqref{initS}, we obtain $\langle S_4(x, \la), C_4(x, \la) \rangle_{|x = 1} = C_4(1,\la)$. Hence $U_1(S_4) = -C_4(1,\la)$. Taking \eqref{D4} into account, we arrive at \eqref{D11}. The relations \eqref{D21} and \eqref{D31} are proved analogously.
\end{proof}

Next, we need estimates for $\Delta_{jk}(\la)$ as $|\la| \to \infty$. Divide the complex $\rho$-plane into the sectors
\begin{equation} \label{defGa}
\Gamma_{\xi} := \left\{ \rho \colon \frac{\pi(\xi-1)}{4} < \arg \rho < \frac{\pi \xi}{4}\right\}, \quad \xi = \overline{1,8}.
\end{equation}
Denote by $\{ \rho_{jk, l} \}_{l \ge 1}$ the zeros of the entire function $\Delta_{jk}(\rho^4)$ in the $\rho$-plane.
Introduce the regions
$$
G_{\de, jk\xi} := \{ \rho \in \overline{\Gamma}_{\xi} \colon |\rho - \rho_{jk,l}| \ge \de, \, l \ge 1 \}, \quad \de > 0.
$$

For a fixed sector $\Gamma_{\xi}$, denote by $\{ \om_l \}_{l = 1}^4$ the roots of the equation $\om^4 = 1$ numbered so that
$$
\mbox{Re} \, (\rho \om_1) < \mbox{Re} \, (\rho \om_2) < \mbox{Re} \, (\rho \om_3) < \mbox{Re} \, (\rho \om_4).
$$

\begin{prop}[\cite{Bond21, Bond22-asympt}] \label{prop:estD}
For each fixed $\xi \in \{ 1, 2, \dots, 8 \}$ and sufficiently large values of $|\rho|$, the following estimates hold:
\begin{align*}
\Delta_{jk}(\la) = O\left( \rho^{a_{jk}} \exp(\rho s_k)\right), \quad \rho \in \Gamma_{\xi}, \\
|\Delta_{jk}(\la)| \ge c_{\de} |\rho|^{a_{jk}} \exp(\mbox{Re}\, (\rho s_k)), \quad \rho \in G_{\de, jk \xi},
\end{align*}
where $\la = \rho^4$, $1 \le k \le j \le 4$, $c_{\de}$ is a constant which depends only on $\de$, and
$$
s_k := \sum\limits_{l = k+1}^4 \om_l, \quad a_{jk} := \sum_{l = 1}^{k-1} p_{l,0} + p_{j,0} + \sum_{l = 1}^{4-k} p_{l,1} - 6.
$$
\end{prop}

Recall that $m_{32}(\la)$ has the simple poles $\{ \la_n \}_{n \ge 1}$. Denote $\be_n := \Res_{\la = \la_n} m_{32}(\la)$. 

\begin{lem} \label{lem:m32}
The function $m_{32}(\la)$ can be uniquely recovered from $\{ \la_n, \be_n\}_{n \ge 1}$ by the formula
\begin{equation} \label{m32}
m_{32}(\la) = \sum_{n = 1}^{\infty} \frac{\be_n}{\la - \la_n},
\end{equation}
where the series converges absolutely and uniformly on compact sets in $\mathbb C \setminus \{ \la_n \}$.
\end{lem}

\begin{proof}
Due to the asymptotical results of \cite{Bond22-asympt, Bond22-alg}, we have $\la_n \sim c n^4$, $\be_n = O(1)$ as $n \to \infty$. This implies the convergence of the series in \eqref{m32}. Next, using \eqref{mjk} and Proposition~\ref{prop:estD}, we obtain the estimate $|m_{32}(\rho^4)| \le c_{\de} |\rho|^{-3}$ for $\rho \in G_{\de,22\xi}$, $\xi = \overline{1,8}$. Hence, Mittag-Leffler's theorem implies the relation \eqref{m32}.
\end{proof}

\section{Proofs of Theorems~\ref{thm:Leib} and~\ref{thm:sd}} \label{sec:gen}

In this section, we obtain the Yurko-type uniqueness results for the problem $\mathcal L$. Theorem~\ref{thm:Leib} is proved by the reduction to Proposition~\ref{prop:uniqM}. In other words, we show that, under the separation condition $(\mathcal S)$, the three functions $m_{k+1,k}(\la)$, $k = 1, 2, 3$, uniquely specify the whole Weyl-Yurko matrix $M(\la)$. The proof of Theorem~\ref{thm:sd} is obtained by the general scheme of the method of spectral mappings. An analogous result is provided in \cite{Bond23-reg} for another type of boundary conditions. Nevertheless, for the sake of completeness, we outline the proof of Theorem~\ref{thm:sd} in this section, focusing on the uniqueness of the coefficients $a$, $b$, and $c$ from the boundary conditions \eqref{bc}.

\begin{proof}[Proof of Theorem~\ref{thm:Leib}]
Consider two problem $\mathcal L$ and $\tilde{\mathcal L} = \mathcal L(\tilde p, \tilde q, \tilde a, \tilde b, \tilde c)$, satisfying the separation condition $(\mathcal S)$.
Suppose that $m_{k+1,k}(\la) \equiv \tilde m_{k+1,k}(\la)$ for $k = 1, 2, 3$. Let us show that $M(\la) \equiv \tilde M(\la)$. 

Let $\la_0$ be a zero of $\Delta_{11}(\la)$ of multiplicity $r_1$. Using \eqref{relm31}, we find
$$
m_{31,\langle -k \rangle}(\la_0) = \sum_{j = k}^{r_1} m_{21,\langle-j\rangle}(\la_0) m_{32,\langle j-k \rangle}(\la_0), \quad k = \overline{1,r_1}.
$$
Hence
$$
m_{31,\langle -k \rangle}(\la_0) = \tilde m_{31,\langle -1 \rangle}(\la_0), \quad k = \overline{1,r_1},
$$
and so the function $(m_{31} - \tilde m_{31})$ is entire in $\la$. Using \eqref{mjk} and the estimates of Proposition~\ref{prop:estD}, we obtain
$$
|(m_{31} - \tilde m_{31})(\rho^4)| \le c_{\de} |\rho|^{-2}, \quad \rho \in G_{\de,11\xi}, \quad |\rho| \to \infty.
$$
Consequently, by Liouville's Theorem, $m_{31}(\la) \equiv \tilde m_{31}(\la)$. Analogously, we get $m_{42}(\la) \equiv \tilde m_{42}(\la)$.

Next, let $\la_0$ be a zero of $\Delta_{kk}(\la)$ of multiplicity $r_k$, $k \in \{ 1, 2, 3 \}$. Similarly to Lemma 2.5.1 in \cite{Yur02}, we obtain the relation
\begin{equation} \label{relPhik}
\Phi_k(x, \la) = \xi_k(x, \la) + \sum_{\nu = 1}^{r_k} \frac{c_{k\nu}}{(\la - \la_0)^{\nu}} \Phi_{k+1}(\la),
\end{equation}
where $\xi_k(x, \la)$ is regular at $\la = \la_0$ and $c_{k\nu} = m_{k+1,k,\langle -\nu \rangle}$. Note that $m_{k+1,k}(\la) \equiv \tilde m_{k+1,k}(\la)$ implies $c_{k\nu} = \tilde c_{k\nu}$. It follows from \eqref{relPhik} that
$$
m_{41}(\la) = U_4(\Phi_1) = U_4(\xi_1) + \sum_{\nu = 1}^{r_1} \frac{c_{1\nu}}{(\la - \la_0)^{\nu}} m_{42}(\la),
$$
so the function $(m_{41} - \tilde m_{41})$ is regular at $\la_0$. Hence, it is entire in $\la$. Proposition~\ref{prop:estD} yields the estimate 
$$
|(m_{41} - \tilde m_{41})(\rho^4)| \le c_{\de} |\rho|^{-1}, \quad \rho \in G_{\de,11\xi}, \quad |\rho| \to \infty.
$$
By Liouville's Theorem, $m_{41}(\la) \equiv \tilde m_{41}(\la)$. 
Thus, $M(\la) \equiv \tilde M(\la)$, which by Proposition \ref{prop:uniqM} implies $(p,q,a,b,c) = (\tilde p, \tilde q, \tilde a, \tilde b, \tilde c)$.
\end{proof}

\begin{proof}[Proof of Theorem~\ref{thm:sd}]
Consider two problems $\mathcal L$ and $\tilde {\mathcal L} = \mathcal L(\tilde p, \tilde q, \tilde a, \tilde b, \tilde c)$, satisfying the hypothesis of Theorem~\ref{thm:sd}. Suppose that $\Lambda = \tilde \Lambda$ and $\mathcal N(\la_0) = \tilde{\mathcal N}(\la_0)$ for all $\la_0 \in \Lambda$.

Introduce the matrix of spectral mappings
\begin{equation} \label{defP}
\mathcal P(x, \la) = \Phi(x, \la) \tilde \Phi^{-1}(x, \la).
\end{equation}

\begin{prop}[{\cite[Lemma~9]{Bond22-alg}}]
The matrix function $\mathcal P(x,\la)$ does not depend on $\la$.
\end{prop}

Thus, $\mathcal P(x, \la) \equiv \mathcal P(x)$, and so \eqref{defP} implies
\begin{equation} \label{relP}
\tilde \Phi(x, \la) \mathcal P(x) = \Phi(x, \la).
\end{equation}
Recall that the matrix functions $\Phi(x, \la)$ and $\tilde \Phi(x, \la)$ satisfy the first-order systems of form \eqref{sys}: 
$$
\Phi'(x, \la) = (F(x) + \Lambda) \Phi(x, \la), \quad \tilde \Phi'(x, \la) = (\tilde F(x) + \Lambda) \tilde \Phi(x, \la), \quad x \in (0,1),
$$
where
\begin{equation} \label{structF}
F(x) = \begin{bmatrix}
            0 & 1 & 0 & 0 \\
            0 & 0 & 1 & 0 \\
            0 & p(x) & 0 & 1 \\
            -q(x) & 0 & 0 & 0
       \end{bmatrix}, \qquad
\tilde F(x) = \begin{bmatrix}
            0 & 1 & 0 & 0 \\
            0 & 0 & 1 & 0 \\
            0 & \tilde p(x) & 0 & 1 \\
            -\tilde q(x) & 0 & 0 & 0
       \end{bmatrix}.
\end{equation}
Therefore, we use \eqref{relP} to derive the relation
\begin{equation} \label{eqP}
\mathcal P'(x) + \mathcal P(x) \tilde F(x) = F(x) \mathcal P(x), \quad x \in (0,1),
\end{equation}
In addition, the relation \eqref{relP} together with the boundary conditions \eqref{bcPhi} imply that $\mathcal P(x)$ is a unit lower-triangular matrix. Under these requirements, the following proposition holds.

\begin{prop}[{\cite[Lemma 2.4]{Bond23-mmas}}] \label{prop:P}
Suppose that the matrices $F(x)$ and $\tilde F(x)$ have the structure \eqref{structF}.
Then, the relation \eqref{eqP} implies that $\mathcal P(x)$ identically equals to the unit matrix and so $p(x) = \tilde p(x)$, $q(x) = \tilde q(x)$ a.e. on $(0,1)$. 
\end{prop}

Thus, we conclude from Proposition~\ref{prop:P} and the relation \eqref{relP} that $\Phi(x, \la) \equiv \tilde \Phi(x, \la)$. The boundary conditions \eqref{bcPhi} for $k = 3, 4$ imply
$$
a = -\Phi_4''(0,\la), \quad b = -\Phi_4^{[3]}(0, \la), \quad c = -(\Phi_3^{[3]} + b \Phi_3')(0,\la).
$$
Since $\Phi_k(x, \la) \equiv \tilde \Phi_k(x, \la)$ for $k = 3, 4$, then $a = \tilde a$, $b = \tilde b$, $c = \tilde c$, which concludes the proof.
\end{proof}

\section{Proof of Theorem~\ref{thm:uniq1}} \label{sec:uniq1}

In this section, we prove the uniqueness of solution for Inverse Problem~\ref{ip:main} under the separation condition $(\mathcal S)$. For this purpose, we establish the uniqueness of the recovery for the functions $m_{k+1,k}(\la)$, $k = 1, 2, 3$, from the spectral data $\{ \la_n, \ga_n, \xi_n \}_{n \ge 1}$, and so deduce Theorem~\ref{thm:uniq1} from Theorem~\ref{thm:Leib}. 

\begin{lem} \label{lem:findbg}
Suppose that $\Delta_{33}(\la_n) \ne 0$ for some $n \ge 1$. Then, the numbers $\be_n$ and $m_{43}(\la_n)$ are uniquely recovered from the spectral data $\la_n, \xi_n, \ga_n$ by the formulas:
\begin{equation} \label{findbg}
\be_n = -\gamma_n^2, \quad m_{43}(\la_n) = \frac{\xi_n}{\ga_n}.
\end{equation}
\end{lem}

\begin{proof}
Consider the function $\Phi_3(x, \la)$. In view of \eqref{relM} and \eqref{mjk}, its poles coincide with the zeros of $\Delta_{33}(\la)$, so $\la_n$ is a regular points of $\Phi_3(x, \la)$.
Therefore, comparing \eqref{bc2} and \eqref{bcPhi}, we conclude that $\Phi_3(x, \la_n)$ is the eigenfunction of the problem $\mathcal L$ corresponding to the eigenvalue $\la_n$. Since
$$
\Phi_3(0, \la_n) = 1, \quad \Phi_3'(0, \la_n) = m_{43}(\la_n),
$$
then $\ga_n \ne 0$, $\Phi_3(x, \la_n) = \frac{1}{\ga_n} y_n(x)$, 
$m_{43}(\la_n) = \frac{\xi_n}{\ga_n}$.

The relation \eqref{wron} together with the boundary conditions \eqref{bcPhi} imply
$$
(\la - \la_n) \int_0^1 \Phi_2(x, \la) \Phi_3(x, \la_n) \, dx = \langle \Phi_2(x, \la), \Phi_3(x, \la_n) \rangle \big|_0^1 = -1.
$$

It follows from \eqref{relM} that $C(x, \la) = \Phi(x, \la) M^{-1}(\la)$. Using \eqref{invM}, we derive
\begin{equation} \label{relCPhi}
C_2(x, \la) = \Phi_2(x, \la) - m_{32}(\la) \Phi_3(x, \la) + m_{31}(\la) \Phi_4(x, \la).
\end{equation}
Recall that $\Phi_2(x, \la)$ and $m_{32}(\la)$ have a simple pole at $\la = \la_n$, and the other functions in \eqref{relCPhi} are entire in view of $\Delta_{11}(\la_n) = -\Delta_{33}(\la_n) \ne 0$.
Consequently, we get from \eqref{relCPhi} that
\begin{equation} \label{ResPhi2}
\Res_{\la = \la_n} \Phi_2(x, \la) = \be_n \Phi_3(x, \la_n).
\end{equation}

Combining \eqref{relCPhi} and \eqref{ResPhi2}, we arrive at the relation
$$
\int_0^1 \Phi_3^2(x, \la_n) \, dx = -\frac{1}{\be_n}.
$$
On the other hand, the equality $\Phi_3(x, \la_n) = \dfrac{1}{\ga_n} y_n(x)$ together with the condition $\int_0^1 y_n^2(x) \, dx = 1$ imply
$$
\int_0^1 \Phi_3^2(x, \la_n) \, dx = \frac{1}{\ga_n^2}.
$$
Hence $\be_n = -\ga_n^2$.
\end{proof}

\begin{remark} \label{rem:yn}
It follows from the proof of Lemma~\ref{lem:findbg} that, if $\la_n$ satisfies the hypothesis of the lemma, then any corresponding eigenfunction fulfills the condition $\int_0^1 y_n^2(x) \, dx \ne 0$.
\end{remark}

\begin{cor} \label{cor:findbg}
Under the separation condition $(\mathcal S)$, the spectral data $\{ \la_n, \ga_n, \xi_n \}_{n \ge 1}$ uniquely determine $\{ \be_n \}_{n \ge 1}$ and $\{ m_{43}(\la_n) \}_{n \ge 1}$ by the formulas \eqref{findbg}.
\end{cor}

\begin{lem} \label{lem:m43}
Under the condition $(\mathcal S)$, the function $m_{43}(\la)$ is uniquely specified by $\{ \la_n \}_{n \ge 1}$ and $\{ m_{43}(\la_n) \}_{n \ge 1}$.
\end{lem}

\begin{proof}
Consider two problems $\mathcal L$ and $\tilde {\mathcal L} = \mathcal L(\tilde p, \tilde q, \tilde a, \tilde b, \tilde c)$. Suppose that they both satisfy $(\mathcal S)$, $\la_n = \tilde \la_n$ and $m_{43}(\la_n) = \tilde m_{43}(\la_n)$ for all $n \ge 1$. Using \eqref{mjk}, we derive
$$
m_{43}(\la_n) - \tilde m_{43}(\la_n) = \frac{\Delta_{33}(\la_n) \tilde \Delta_{43}(\la_n) - \tilde \Delta_{33}(\la_n) \Delta_{43}(\la_n)}{\Delta_{33}(\la_n) \tilde \Delta_{33}(\la_n)} = 0, \quad n \ge 1.
$$
Hence, the function $G(\la) := \Delta_{33}(\la) \tilde \Delta_{43}(\la) - \tilde \Delta_{33}(\la) \Delta_{43}(\la)$ has zeros $\{ \la_n \}_{n \ge 1}$, so the function $K(\la) := \dfrac{G(\la)}{\Delta_{22}(\la)}$ is entire in $\la$. Applying Proposition~\ref{prop:estD}, we obtain the estimate
$$
|K(\rho^4)| = O\left( \rho^{-1} \exp(\mbox{Re} \, (\rho(\om_4 - \om_3)))\right), \quad |\rho| \to \infty, \quad \rho \in \overline{\Gamma}_{\xi}, \quad \xi = \overline{1,8}.
$$
In particular, for $\arg \rho = \frac{\pi}{4}$, we have $\mbox{Re}(\rho(\om_4 - \om_3)) = 0$ ($\xi = 1$ or $2$). Therefore, $K(\la)$ has the order $\frac{1}{4}$ and $K(\la) \to 0$ as $\la \to -\infty$, $\la \in \mathbb R$. Consequently, by Phragmen-Lindel\"of's theorem (see \cite[Corollary~5.1]{BFY14}), we conclude that $K(\la) \equiv 0$. Obviously, this implies $m_{43}(\la) \equiv \tilde m_{43}(\la)$.
\end{proof}

Thus, by virtue of Corollary~\ref{cor:findbg} and Lemmas \ref{lem:m32} and \ref{lem:m43}, the spectral data $\{ \la_n, \ga_n, \xi_n \}_{n \ge 1}$ uniquely specify the functions $m_{k+1,k}(\la)$, $k = 1, 2, 3$. Hence, Theorem~\ref{thm:Leib} readily implies Theorem~\ref{thm:uniq1}.

\begin{remark}
The central place in the proofs of Theorem~\ref{thm:uniq1} is taken by Lemma~\ref{lem:m43}. It is based on the construction of an entire function $K(\la)$ of order $\frac{1}{4}$, which tends to zero as $\la \to -\infty$ along the real line. This idea does not work for differential equations of even order higher than four. Thus, generalization of McLaughlin's problem to orders higher than four is an open question.
\end{remark}

\begin{remark}
In this case, we confine ourselves to the case of simple eigenvalues $\{ \la_n \}_{n \ge 1}$. In general, an eigenvalue of $\mathcal L$ can have two linearly independent eigenfunctions and/or associated functions. In the case of associated functions, the weight numbers $\ga_n$ and $\xi_n$ have to be defined in another way, e.g., analogously to the studies \cite{But07, BSY13} for the second-order equation \eqref{StL}. However, for the order four, the non-simplicity of the spectrum leads to significant technical difficulties, so we exclude this case from consideration.
\end{remark}

\section{Proof of Theorem~\ref{thm:uniq2}} \label{sec:uniq2}

In this section, we prove the second uniqueness theorem for McLaughlin's problem by obtaining the weight matrices $\{ \mathcal N(\la_0) \}_{\la_0 \in \Lambda}$ from the spectral data $\{ \la_n, \ga_n, \xi_n \}_{n \ge 1}$. For these purpose, we consider separately the cases (I)--(V) from \eqref{cases}. The proof relies on several auxiliary lemmas. The central part is taken by Lemma~\ref{lem:weight}, which presents the structure of the weight matrices for all the five cases. Furthermore, it is a challenge to determine for each given eigenvalue $\la_n$ to which case does it belong. A solution to this problem is provided by Algorithm~\ref{alg:case}.  In Lemmas~\ref{lem:case2}, \ref{lem:recD}, and \ref{lem:cases34}, we derive relations between the data $\{ \la_n, \ga_n, \xi_n \}_{n \ge 1}$ and non-zero elements of the weight matrices $\{ \mathcal N(\la_n) \}_{n \ge 1}$ and, finally, conclude the proof of Theorem~\ref{thm:uniq2}. 

\begin{lem} \label{lem:aux2}
In the cases (I)-(IV) of \eqref{cases}, the following relations hold:
\begin{align*}
(I) \colon & \quad C_4(1, \la_n) \ne 0, \\
(II) \colon & \quad C_4(1, \la_n) = C_3(1, \la_n) = 0, \quad C_4'(1, \la_n) \ne 0, \\
(III) \colon & \quad C_4(1, \la_n) = C_4'(1, \la_n) = 0, \quad C_3(1, \la_n) \ne 0, \\
(IV) \colon & \quad C_4(1, \la_n) = C_3(1, \la_n) = C_4'(1, \la_n) = 0.
\end{align*}
\end{lem}

\begin{proof}
Suppose that $C_4(1, \la_n) = 0$. Then, since $\Delta_{22}(\la_n) = 0$, we conclude from the formula for $\Delta_{22}(\la)$ in \eqref{D3} that $C_3(1, \la_n) = 0$ or $C_4'(1, \la_n) = 0$. If $C_3(1, \la_n) = 0$, then Lemma~\ref{lem:aux1} implies that $U_1(S_4) = U_2(S_4) = 0$ at $\la = \la_n$. Hence $S_4(x, \la_n)$ is the eigenfunction corresponding to $\la_n$. Since $\la_n$ is a simple eigenvalue, then this eigenfunction is unique up to a constant multiplier. Hence $V_3(y_n) = 0$, so we have either the case (II) or the case (IV). If $C_4'(1, \la_n) = 0$, then $C_4(x, \la_n)$ is the eigenfunction, so $U_3(y_n) = 0$. This corresponds to the cases (III) and (IV). Consequently, $C_4(1, \la_n) = 0$ implies $U_3(y_n) = 0$ or $V_3(y_n) = 0$. Hence, $C_4(1, \la_n) \ne 0$ in the case (I). This concludes the proof.
\end{proof}

Lemmas \ref{lem:aux1} and \ref{lem:aux2} imply the following corollary.

\begin{cor} \label{cor:sep}
In the case (I) of \eqref{cases}, the eigenvalue $\la_n$ satisfies the separation condition, that is, $\Delta_{11}(\la_n) \ne 0$ and $\Delta_{33}(\la_n) \ne 0$. In the cases (II)-(IV), on the contrary, $\Delta_{11}(\la_n) = \Delta_{33}(\la_n) = 0$.
\end{cor}

\begin{lem} \label{lem:weight}
Suppose that $\Delta_{11}(\la)$ and $\Delta_{22}(\la)$ have only simple zeros. Then, the weight matrices in the cases (I)--(IV) of \eqref{cases} have the following structure:
\begin{align} \nonumber
(I) \colon & \quad \mathcal N(\la_n) = \begin{bmatrix}
                                            0 & 0 & 0 & 0 \\
                                            0 & 0 & 0 & 0 \\
                                            0 & n_{32}(\la_n) & 0 & 0 \\
                                            0 & 0 & 0 & 0 
                                       \end{bmatrix}, \quad n_{32}(\la_n) = \be_n \ne 0, \\ \nonumber
(II) \colon & \quad \mathcal N(\la_n) = \begin{bmatrix}
                                            0 & 0 & 0 & 0 \\
                                            0 & 0 & 0 & 0 \\
                                            n_{31}(\la_n) & n_{32}(\la_n) & 0 & 0 \\
                                            n_{41}(\la_n) & n_{42}(\la_n) & 0 & 0 
                                        \end{bmatrix}, \quad
                                        \begin{vmatrix}
                                            n_{31} & n_{32} \\
                                            n_{41} & n_{42}
                                        \end{vmatrix}(\la_n) = 0, 
                                        \quad n_{31} n_{32} n_{41} n_{42} \ne 0, \\ \label{rel2}
                                        & 
                                        n_{31}(\la_n) = -n_{42}(\la_n) = m_{31,\langle -1 \rangle}(\la_n), \quad
                                        n_{41}(\la_n) = m_{41,\langle-1\rangle}(\la_n) - m_{43}(\la_n) n_{31}(\la_n), \\ \nonumber
(III) \colon & \quad \mathcal N(\la_n) = \begin{bmatrix}
                                            0 & 0 & 0 & 0 \\
                                            n_{21}(\la_n) & 0 & 0 & 0 \\
                                            0 & 0 & 0 & 0 \\
                                            n_{41}(\la_n) & 0 & n_{43}(\la_n) & 0
                                        \end{bmatrix}, 
                                        \quad n_{21}(\la_n) = n_{43}(\la_n) = m_{43,\langle -1 \rangle}(\la_n), \\ \label{rel3} & \qquad \qquad n_{41}(\la_n) = m_{41,\langle -1 \rangle}(\la_n) - m_{42}(\la_n) m_{21,\langle -1 \rangle}(\la_n),         
                                        \\ \nonumber
(IV) \colon & \quad \mathcal N(\la_n) = \begin{bmatrix}
                                            0 & 0 & 0 & 0 \\
                                            0 & 0 & 0 & 0 \\
                                            0 & 0 & 0 & 0 \\
                                            n_{41}(\la_n) & 0 & 0 & 0 
                                        \end{bmatrix}, 
                                        \quad n_{41}(\la_n) = m_{41,\langle -1 \rangle}(\la_n), \\
                                        \nonumber
(V) \colon & \quad \mathcal N(\la_0) = \begin{bmatrix}
                                            0 & 0 & 0 & 0 \\
                                            n_{21}(\la_0) & 0 & 0 & 0 \\
                                            0 & 0 & 0 & 0 \\
                                            0 & 0 & n_{43}(\la_0) & 0
                                        \end{bmatrix}, \quad
                                        n_{21}(\la_0) = n_{43}(\la_0) = m_{43,\langle -1 \rangle}(\la_0) \ne 0.
\end{align}
\end{lem}

\begin{proof}
Let $\la_0$ be a pole of $M(\la)$. Then, the definition \eqref{defN} of the weight matrices implies
$$
M_{\langle 0 \rangle}(\la_0) \mathcal N(\la_0) = M_{\langle -1 \rangle}(\la_0).
$$
Taking the structure of the Weyl matrix \eqref{matrM} into account, we can write in the element-wise form that
$$
\begin{bmatrix}
1 & 0 & 0 & 0 \\
m_{21, \langle 0 \rangle} & 1 & 0 & 0 \\
m_{31, \langle 0 \rangle} & m_{32, \langle 0 \rangle} & 1 & 0 \\
m_{41, \langle 0 \rangle} & m_{42, \langle 0 \rangle} & m_{43, \langle 0 \rangle} & 1
\end{bmatrix}
\begin{bmatrix}
0 & 0 & 0 & 0 \\
n_{21} & 0 & 0 & 0 \\
n_{31} & n_{32} & 0 & 0 \\
n_{41} & n_{42} & n_{43} & 0
\end{bmatrix} = 
\begin{bmatrix}
0 & 0 & 0 & 0 \\
m_{21, \langle -1 \rangle} & 0 & 0 & 0 \\
m_{31, \langle -1 \rangle} & m_{32, \langle -1 \rangle} & 0 & 0 \\
m_{41, \langle -1 \rangle} & m_{42, \langle -1 \rangle} & m_{43, \langle -1 \rangle} & 0
\end{bmatrix}.
$$
Here and below, the argument $\la_0$ is omitted for brevity. Then, we find
\begin{gather} \label{n21}
n_{21} = m_{21, \langle -1 \rangle}, \quad n_{32} = m_{32, \langle -1 \rangle}, \quad n_{43} = m_{43,\langle -1 \rangle}, \\ \label{n31}
n_{31} = m_{31, \langle -1 \rangle} - m_{32,\langle 0 \rangle} m_{21, \langle -1 \rangle}, \quad
n_{42} = m_{42,\langle -1 \rangle} - m_{43,\langle 0 \rangle} m_{32,\langle -1 \rangle}, \\ \label{n41}
n_{41} = m_{41, \langle -1 \rangle} - m_{42,\langle 0 \rangle} n_{21} - m_{43,\langle 0 \rangle} n_{31}.
\end{gather}

Since $m_{21}(\la) \equiv m_{43}(\la)$, the relations \eqref{n21} imply $n_{21} = n_{43}$.
Next, calculating the residue of the left-hand side in \eqref{relm31} at $\la = \la_0$, we obtain
\begin{equation} \label{aux1}
m_{31, \langle -1 \rangle} - m_{21, \langle -1 \rangle} m_{32, \langle 0 \rangle} - m_{21,\langle 0 \rangle} m_{32, \langle -1 \rangle} + m_{42, \langle -1 \rangle} = 0.
\end{equation}

Using \eqref{aux1} together with \eqref{n31} and taking the relation $m_{21}(\la) \equiv m_{43}(\la)$ into account, we conclude that $n_{31} = -n_{42}$.

Now, let us separately consider the cases (I)--(V).

\smallskip

\textit{Case (I).} Due to Corollary~\ref{cor:sep}, $\Delta_{11}(\la_n) \ne 0$ and $\Delta_{33}(\la_n) \ne 0$. Therefore, in view of \eqref{mjk}, we have $m_{j1, \langle -1 \rangle}(\la_n) = 0$ for $j = 2, 3, 4$ and $m_{43, \langle -1 \rangle}(\la_n) = 0$. Then, the relations \eqref{n21}--\eqref{n41} imply that $n_{j1} = 0$ for $j = 2, 3, 4$ and $n_{43} = 0$. Since $n_{42} = -n_{31}$, we also have $n_{42} = 0$. Thus, there is the only non-zero element $n_{32} = m_{32,\langle -1 \rangle}$.

\smallskip

\textit{Case (II).} By Lemma~\ref{lem:aux2}, we have 
\begin{equation} \label{sm2}
C_3(1,\la_n) = C_4(1,\la_n) = 0. 
\end{equation}
Therefore, Lemma~\ref{lem:aux1} implies that 
$$
\Delta_{11}(\la_n) = \Delta_{21}(\la_n) = \Delta_{33}(\la_n) = \Delta_{43}(\la_n) = 0.
$$ 
Since $\Delta_{11}(\la)$ and $\Delta_{33}(\la)$ have only simple zeros, we get from \eqref{mjk} that $m_{21,\langle -1 \rangle}(\la_n) = m_{43,\langle -1 \rangle}(\la_n) = 0$. Together with \eqref{n21}, this yields $n_{21} = n_{43} = 0$. Therefore, the relations \eqref{n31} and \eqref{n41} imply \eqref{rel2}.

Let us show that
\begin{equation} \label{det0}
\begin{vmatrix}
n_{31} & n_{32} \\
n_{41} & n_{42}
\end{vmatrix} = 0.
\end{equation}
Using \eqref{n31} and \eqref{n41}, we derive
\begin{equation} \label{sm3}
\begin{vmatrix}
n_{31} & n_{32} \\
n_{41} & n_{42}
\end{vmatrix} = m_{31,\langle -1 \rangle} m_{42, \langle -1 \rangle} - m_{41,\langle -1 \rangle} m_{32, \langle -1 \rangle}.
\end{equation}
Using \eqref{D2}--\eqref{D4} and taking \eqref{sm2} into account, we obtain
\begin{multline} \label{sm4}
\Delta_{31}(\la_n) \Delta_{42}(\la_n) - \Delta_{41}(\la_n) \Delta_{32}(\la_n) =  
\left( C_4'' \begin{vmatrix} C_2' & C_1' \\ C_2 & C_1 \end{vmatrix} - C_4' \begin{vmatrix} C_2'' & C_1'' \\ C_2 & C_1 \end{vmatrix} \right) C_3' C_2 \\  - 
\left( C_3' \begin{vmatrix} C_2'' & C_1'' \\ C_2 & C_1 \end{vmatrix} - C_3'' \begin{vmatrix} C_2' & C_1' \\ C_2 & C_1 \end{vmatrix}\right) (-C_2 C_4') 
= C_2 \begin{vmatrix} C_4'' & C_3'' \\ C_4' & C_3' \end{vmatrix} \begin{vmatrix} C_2' & C_1' \\ C_2 & C_1 \end{vmatrix},
\end{multline}
where the arguments $(1, \la_n)$ are omitted for brevity. Consider the Lagrange bracket
$$
\langle C_3, C_4 \rangle = C_3^{[3]} C_4 - C_3'' C_4' + C_3' C_4'' - C_3 C_4^{[3]}.
$$
The identity \eqref{wron} implies that 
$$
\langle C_3(x, \la_n), C_4(x, \la_n) \rangle_{|x = 0} = \langle C_3(x, \la_n), C_4(x, \la_n) \rangle_{|x = 1}.
$$
Using \eqref{initC} and \eqref{sm2}, we obtain
$$
\langle C_3(x, \la_n), C_4(x, \la_n) \rangle_{|x = 0} = 0, \quad
\langle C_3(x, \la_n), C_4(x, \la_n) \rangle_{|x = 1} = \begin{vmatrix} C_4'' & C_3'' \\ C_4' & C_3' \end{vmatrix}(1, \la_n).
$$
Hence, the relation \eqref{sm4} implies $(\Delta_{31} \Delta_{42} - \Delta_{41}\Delta_{32})(\la_n) = 0$. Therefore, the right-hand side of \eqref{sm3} turns into zero, so we arrive at \eqref{det0}.

Thus, it remains to show that all the elements of the determinant \eqref{det0} are non-zero. Suppose that $n_{31} = -n_{42} = 0$. Then, either $n_{41} = 0$ or $n_{32} = 0$. But this corresponds either to the case (I) or to the case (IV), respectively. Hence $n_{31} \ne 0$, $n_{42} \ne 0$, which implies $n_{41} n_{32} \ne 0$. This concludes the proof for the case (II).

\smallskip

\textit{Case (III).} Due to Lemma~\ref{lem:aux2}, we have $C_4(1, \la_n) = C_4'(1,\la_n) = 0$. Therefore, it follows from \eqref{D1}, \eqref{D3}, and \eqref{D11} that $\Delta_{22}(\la_n) = \Delta_{33}(\la_n) = 0$ and
\begin{equation} \label{sm5}
\Delta_{11}(\la_n) = C_4''(1,\la_n) 
\begin{vmatrix} 
C_2'(1,\la_n) & C_3'(1,\la_n) \\ 
C_2(1, \la_n) & C_3(1, \la_n)
\end{vmatrix} = 0.
\end{equation}
If $C_4''(1, \la_n) = 0$, then $\Delta_{21}(\la_n) = 0$. In view of \eqref{D21}, we obtain $C_3(1, \la_n) = 0$, so we arrive at the case (IV). Therefore, in the case (III), $C_4''(1,\la_n) \ne 0$, so the determinant in \eqref{sm5} equals zero. This implies $\Delta_{42}(\la_n) = 0$. Consequently, $m_{32,\langle -1 \rangle}(\la_n) = m_{42,\langle -1 \rangle}(\la_n) = 0$. Hence, it follows from \eqref{n21} and \eqref{n31} that $n_{32} = n_{42} =0$. Furthermore, $n_{31} = -n_{42} = 0$. Then, the relation \eqref{rel3} readily follows from \eqref{n41}.

\smallskip

\textit{Case (IV).} Due to Lemma~\ref{lem:aux2}, we have $C_4(1, \la_n) = C_3(1, \la_n) = C_4'(1, \la_n) = 0$. As in the cases (II) and (III), one can show that all the elements of $\mathcal N(\la_n)$ except for $n_{41}$ equal zero. Thus, the relation \eqref{n41} implies $n_{41} = m_{41,\langle -1 \rangle}$.

\smallskip 

\textit{Case (V).} It follows from Lemma~4 in \cite{Bond22-alg} that, 
if $\la_0$ is not a zero of two neighbouring characteristic functions $\Delta_{kk}(\la)$ and $\Delta_{k+1,k+1}(\la)$, then the corresponding weight matrix has the following structure:
$$
\mathcal N(\la_0) = \begin{bmatrix}
                        0 & 0 & 0 & 0 \\
                        n_{21} & 0 & 0 & 0 \\
                        0 & n_{32} & 0 & 0 \\
                        0 & 0 & n_{43} & 0
                    \end{bmatrix}.
$$

In the case (V), we have $\Delta_{11}(\la_0) = \Delta_{33}(\la_0) = 0$ and $\Delta_{22}(\la_0) \ne 0$, so \eqref{mjk} and \eqref{n21} imply that $n_{32} = 0$. It has been already proved that the elements $n_{21}$ and $n_{43}$ are equal to each other and can be found from \eqref{n21}.
This completes the proof.
\end{proof}

\begin{lem} \label{lem:recm43}
Under the hypothesis of Theorem~\ref{thm:uniq2}, the spectral data $\{ \la_n, \ga_n, \xi_n \}_{n \ge 1}$ uniquely specify the function $m_{43}(\la)$.
\end{lem}

\begin{proof}
Consider two problems $\mathcal L$ and $\tilde{\mathcal L} = \mathcal L(\tilde p, \tilde q, \tilde a, \tilde b, \tilde c)$. Suppose that they both fulfill the hypothesis of Theorem~\ref{thm:uniq2} and $\la_n = \tilde \la_n$, $\ga_n = \tilde \ga_n$, $\xi_n = \tilde \xi_n$, $n \ge 1$.

Note that, if $\ga_n \ne 0$, then $\la_n$ is of the case (I) or (II), and if $\ga_n = 0$, then $\la_n$ is of the case (III) or (IV). Analogously to the proof of Lemma~\ref{lem:m43}, consider the function 
$$
G(\la) = \Delta_{33}(\la) \tilde \Delta_{43}(\la) - \tilde \Delta_{33}(\la) \Delta_{43}(\la).
$$
If $\la_n = \tilde \la_n$ is of the case (I) for the both problems $\mathcal L$ and $\tilde{\mathcal L}$, then one can show that $G(\la_n) = 0$ similarly to the proof of Lemma~\ref{lem:m43}. In the cases (II) and (IV), Lemmas~\ref{lem:aux1} and~\ref{lem:aux2} imply $\Delta_{33}(\la_n) = \Delta_{43}(\la_n) = 0$. Hence, if $\la_n$ is of case (II) or (IV) for $\mathcal L$ or $\tilde{\mathcal L}$, then $G(\la_n) = 0$. If $\la_n$ is of the case (III) for the both problems, then $\Delta_{33}(\la_n) = \tilde \Delta_{33}(\la_n) = 0$, which also implies $G(\la_n) = 0$. Thus, $G(\la_n) = 0$ in all the possible cases. Repeating the remaining part of the proof of Lemma~\ref{lem:m43}, we conclude that $m_{43}(\la) \equiv \tilde m_{43}(\la)$.
\end{proof}

\begin{lem} \label{lem:case2}
In the case (II), there exists a constant $\al_n \ne 0$ such that $S_4(x, \la_n) = \al_n y_n(x)$. Furthermore, the following relation holds:
\begin{equation} \label{findal}
\ga_n \dot \Delta_{43}(\la_n) + \xi_n \dot \Delta_{33}(\la_n) = -\al_n,
\end{equation}
where $\dot \Delta_{jk}(\la) = \frac{d}{d\la} \Delta_{jk}(\la)$.
\end{lem}

\begin{proof}
In the case (II), $S_4(x, \la_n)$ is the eigenfunction corresponding to the eigenvalue $\la_n$, so $S_4(x, \la_n) = \al_n y_n(x)$, where
\begin{equation} \label{defal}
\alpha_n^2 = \int_0^1 S_4^2(x, \la_n) \, dx.
\end{equation}
Hence
\begin{equation} \label{S4}
S_4(0, \la_n) = \al_n \ga_n, \quad S_4'(0, \la_n) = \al_n \xi_n. 
\end{equation}

Using \eqref{wron}, \eqref{brack0}, and \eqref{initS}, we derive
\begin{multline} \label{sm6}
(\la - \la_n) \int_0^1 S_4(x, \la) S_4(x, \la_n) \, dx = \langle S_4(x, \la), S_4(x, \la_n) \rangle \big|_0^1 = -U_2(S_4(x, \la)) U_3(S_4(x, \la_n)) \\ + U_1(S_4(x, \la)) U_4(S_4(x, \la_n)) - U_4(S_4(x, \la)) U_1(S_4(x, \la_n)) + U_3(S_4(x, \la)) U_2(S_4(x, \la_n)).
\end{multline}
By Lemma~\ref{lem:aux2}, the relations $C_4(1,\la_n) = C_3(1,\la_n) = 0$ hold in the case (II). In view of Lemma~\ref{lem:aux1}, this yields $U_1(S_4(x, \la_n)) = U_2(S_4(x, \la_n)) = 0$. Consequently, it follows from \eqref{sm6} and Lemma~\ref{lem:aux1} that
$$
(\la - \la_n) \int_0^1 S_4(x, \la) S_4(x, \la_n) \, dx = - (\Delta_{43}(\la) S_4(0, \la_n) + \Delta_{33}(\la) S_4'(0, \la_n)).
$$
Dividing the both sides of the latter relation by $(\la - \la_n)$, passing to the limit as $\la \to \la_n$, and using \eqref{defal} and \eqref{S4}, we obtain
$$
\al_n^2 = -(\al_n \ga_n \dot \Delta_{43}(\la_n) + \al_n \xi_n \dot \Delta_{33}(\la_n)),
$$
which implies \eqref{findal}.
\end{proof}

Given the spectral data $\{ \la_n, \ga_n, \xi_n \}_{n \ge 1}$ and the function $m_{43}(\la)$, one can uniquely determine for each eigenvalue $\la_n$ to which case it belongs by using the following algorithm.

\begin{alg} \label{alg:case}
Let the data $\la_n, \ga_n, \xi_n$ for a fixed $n \ge 1$ and the function $m_{43}(\la)$ be given.
\begin{enumerate}
\item If $\ga_n = 0$ and $\la_n$ is a pole of $m_{43}(\la)$, then $\la_n$ belongs to the case (III).
\item If $\ga_n = 0$ and $\la_n$ is not a pole of $m_{43}(\la)$, then $\la_n$ belongs to the case (IV).
\item If $\ga_n \ne 0$ and $m_{43}(\la_n) = \frac{\xi_n}{\ga_n}$, then $\la_n$ belongs to the case (I).
\item Otherwise, $\la_n$ belongs to the case (II).
\end{enumerate}
\end{alg}

Indeed, the relation $m_{43}(\la_n) = \dfrac{\xi_n}{\ga_n}$ in the case (I) holds by Lemma~\ref{lem:findbg}. In the case (II), the both functions $\Delta_{33}(\la)$ and $\Delta_{43}(\la)$ have a zero $\la_n$, so $m(\la_n) = -\dfrac{\dot\Delta_{43}(\la_n)}{\dot \Delta_{33}(\la_n)}$. This value cannot be equal to $\dfrac{\xi_n}{\gamma_n}$ by virtue of Lemma~\ref{lem:case2}. These arguments justify the steps 3 and 4 of Algorithm~\ref{alg:case}.

\begin{lem} \label{lem:recD}
The spectral data $\{ \la_n, \ga_n, \xi_n \}_{n \ge 1}$ uniquely specify the functions $\Delta_{33}(\la)$ and $\Delta_{43}(\la)$.
\end{lem}

\begin{proof}
Recall that $m_{43}(\la) = -\dfrac{\Delta_{43}(\la)}{\Delta_{33}(\la)}$. Therefore, distinct zeros of the entire functions $\Delta_{33}(\la)$ and $\Delta_{43}(\la)$ can be found as the poles and the zeros of $m_{43}(\la)$, respectively. However, $\Delta_{33}(\la)$ and $\Delta_{43}(\la)$ can have common zeros. In view of Lemma~\ref{lem:aux2}, these common zeros coincide with the eigenvalues $\{ \la_n \}$ that belong to the cases (II) and (IV). Then, the functions $\Delta_{j3}(\la)$, $j = 3,4$ can be constructed by their zeros $\{\la_{n,j3}(\la)\}$ by using Hadamard's Factorization Theorem:
$$
\Delta_{j3}(\la) = c_{j3} \prod_{n = 1}^{\infty} \left(1 - \frac{\la}{\la_{n,j3}} \right), \quad j = 3, 4.
$$
The constants $c_{j3}$ can be found in the standard way (see \cite[Theorem 1.1.4]{FY01}).
\end{proof}

\begin{lem} \label{lem:cases34}
In the cases (III) and (IV), $n_{41}(\la_n) = \xi_n^2$.
\end{lem}

\begin{proof}
In the cases (III) and (IV), we have $\ga_n = 0$ and $y_n(x) = \xi_n C_4(x, \la_n)$. Hence
\begin{equation} \label{intC4}
\int_0^1 C_4^2(x, \la_n) \, dx = \frac{1}{\xi_n^2}.
\end{equation}

\textit{Case (III).} Due to Lemma~\ref{lem:aux2} and \eqref{sm5}, we have
\begin{equation} \label{C4}
C_4(1, \la_n) = C_4'(1, \la_n) = 0, \quad 
\Delta_{42}(\la_n) = \begin{vmatrix}
C_3'(1, \la_n) & C_2'(1, \la_n) \\ C_3(1, \la_n) & C_2(1, \la_n)
\end{vmatrix} = 0.
\end{equation}
Using \eqref{initC}, \eqref{wron}, and \eqref{C4}, we derive
\begin{align*}
(\la - \la_n) \int_0^1 C_4(x, \la) C_4(x, \la_n) \, dx & = \langle C_4(x, \la), C_4(x, \la_n) \rangle \big|_0^1 \\ & = C_4'(1, \la) C_4''(1, \la_n) - C_4(1, \la) C_4^{[3]}(1, \la_n).
\end{align*}
Then, dividing by $(\la - \la_n)$, passing to the limit as $\la \to \la_n$, and using \eqref{intC4}, we arrive at the relation
\begin{equation} \label{sm7}
\dot C_4'(1, \la_n) C_4''(1, \la_n) - \dot C_4(1, \la_n) C_4^{[3]}(1, \la_n) = \frac{1}{\xi_n^2}.
\end{equation}

Next, the formulas \eqref{D3} and \eqref{C4} imply
\begin{equation} \label{sm8}
\dot\Delta_{22}(\la_n) = C_3'(1, \la_n) \dot C_4(1, \la_n) - C_3(1, \la_n) \dot C_4'(1, \la_n).
\end{equation}

Using \eqref{initC}, \eqref{wron}, and \eqref{C4}, we get
\begin{equation} \label{sm9}
\langle C_3(x, \la_n), C_4(x, \la_n) \rangle\big|_0^1 = C_3'(1, \la_n) C_4''(1, \la_n) - C_3(1, \la_n) C_4^{[3]}(1, \la_n) = 0.
\end{equation}

Combining \eqref{sm7}, \eqref{sm8}, and \eqref{sm9}, we deduce
\begin{equation} \label{C4pp}
\frac{\dot\Delta_{22}(\la_n) C_4''(1, \la_n)}{C_3(1, \la_n)} = -\xi_n^2.
\end{equation}

It follows from \eqref{D3}, \eqref{D4}, and \eqref{C4} that $\Delta_{22}(\la_n) = \Delta_{42}(\la_n) = 0$, so $m_{42}(\la_n) = -\dfrac{\dot \Delta_{42}(\la_n)}{\dot \Delta_{22}(\la_n)}$.
Thus, the relation \eqref{rel3} implies
\begin{equation} \label{smn41}
n_{41}(\la_n) =  -\frac{\dot\Delta_{41}(\la_n) \dot \Delta_{22}(\la_n) + \dot \Delta_{42}(\la_n) \Delta_{21}(\la_n)}{\dot\Delta_{11}(\la_n) \dot \Delta_{22}(\la_n)}.
\end{equation}

Let us represent the functions participating in the right-hand side of \eqref{smn41} by using \eqref{D1}--\eqref{D4} and \eqref{C4}:
\begin{align*}
\Delta_{41}(\la_n) & = C_2'' \begin{vmatrix} C_3' & C_1' \\ C_3 & C_1 \end{vmatrix} - C_3'' \begin{vmatrix} C_2' & C_1'\\ C_2 & C_1 \end{vmatrix}, \\
\dot \Delta_{11}(\la_n) & = C_2'' \dot \Delta_{22}(\la_n) - C_3'' \dot \Delta_{32}(\la_n) - C_4'' \dot \Delta_{42}(\la_n), \\ 
\Delta_{21}(\la_n) & = C_4'' \begin{vmatrix}
C_1' & C_3' \\ C_1 & C_3
\end{vmatrix},
\end{align*}
where the arguments $(1, \la_n)$ are omitted for brevity.
Substituting the latter relations into \eqref{smn41}, we obtain
$$
n_{41}(\la_n) = -\frac{1}{\dot \Delta_{11}(\la_n) \dot \Delta_{22}(\la_n)} \biggl( (\dot \Delta_{11}(\la_n) + C_3'' \dot \Delta_{32}(\la_n)) \begin{vmatrix} C_3' & C_1' \\ C_3 & C_1\end{vmatrix} - C_3'' \dot \Delta_{22}(\la_n) \begin{vmatrix} C_2' & C_1' \\ C_2 & C_1 \end{vmatrix}\biggr)
$$
Using \eqref{D4} and \eqref{C4}, one can easily show that
$$
\dot \Delta_{32}(\la_n) \begin{vmatrix} C_3' & C_1' \\ C_3 & C_1 \end{vmatrix} - \dot \Delta_{22}(\la_n) \begin{vmatrix} C_2' & C_1' \\ C_2 & C_1 \end{vmatrix} = 0.
$$
Hence
$$
n_{41}(\la_n) = -\frac{1}{\dot \Delta_{22}(\la_n)} \begin{vmatrix} C_3' & C_1' \\ C_3 & C_1 \end{vmatrix} = \frac{\Delta_{21}(\la_n)}{C_4''(1, \la_n) \dot \Delta_{22}(\la_n)}.
$$
Taking \eqref{C4pp} and \eqref{D21} into account, we arrive at the relation $n_{41}(\la_n) = \xi_n^2$.

\smallskip

\textit{Case (IV)}. Clearly, 
\begin{equation} \label{yncase4}
y_n(x) = \xi_n C_4(x, \la_n) = \theta_n S_4(x, \la_n), 
\end{equation}
where $\theta_n \ne 0$ is a constant, and 
\begin{equation} \label{cond4}
C_4^{(j)}(1, \la_n) = 0, \quad j = 0,1,2, \quad U_s(S_4) = 0, \quad s = 1, 2, 3.
\end{equation}
Using \eqref{initC}, \eqref{wron}, \eqref{D1}, \eqref{initS}, and \eqref{cond4},  we obtain
$$
(\la - \la_n) \int_0^1 S_4(x, \la) C_4(x, \la_n) \, dx = \langle S_4(x, \la), C_4(x, \la_n) \rangle \big|_0^1 = U_1(S_4(x, \la)) = \Delta_{11}(\la).
$$
Taking \eqref{yncase4} into account, we calculate 
\begin{equation} \label{sm10}
\dot \Delta_{11}(\la_n) = \frac{\xi_n}{\theta_n} \int_0^1 C_4^2(x, \la_n) \, dx = \frac{1}{\xi_n \theta_n}. 
\end{equation}
In addition, we get from \eqref{D31} that 
\begin{equation} \label{sm11}
\Delta_{41}(\la_n) = -S_4'(0, \la_n) = -\frac{\xi_n}{\theta_n} C_4'(0, \la_n) = -\frac{\xi_n}{\theta_n}.
\end{equation}
By virtue of \eqref{mjk} and Lemma~\ref{lem:weight}, $n_{41}(\la_n) = -\dfrac{\Delta_{41}(\la_n)}{\dot \Delta_{11}(\la_n)}$. Substitution of \eqref{sm10} and \eqref{sm11} into the latter relation yields the claim. 
\end{proof}

\begin{proof}[Proof of Theorem~\ref{thm:uniq2}]
Let us show that the spectral data $\{ \la_n, \xi_n, \ga_n \}_{n \ge 1}$ uniquely specify the poles $\Lambda$ of the Weyl-Yurko matrix and the weight matrices $\{\mathcal N(\la_0) \}_{\la_0 \in \Lambda}$. Then, Theorem~\ref{thm:sd} will imply the uniqueness for solution of Inverse Problem~\ref{ip:main}.

By Lemma~\ref{lem:recm43}, the spectral data $\{ \la_n, \ga_n, \xi_n \}_{n \ge 1}$ uniquely specify $m_{43}(\la)$. Denote by $\{\mu_n \}_{n \ge 1}$ the poles of $m_{43}(\la)$ that are distinct from the eigenvalues $\{ \la_n \}_{n \ge 1}$. Then, the set $\Lambda = \{ \la_n \}_{n \ge 1} \cup \{ \mu_n \}_{n \ge 1}$ is uniquely found. Furthermore, using Algorithm~\ref{alg:case}, we can determine to which case among (I)--(IV) each eigenvalue $\la_n$ belongs. Let us consider the cases (I)--(V) separately and show that, in each case, the corresponding weight matrix can be found in accordance with its structure from Lemma~\ref{lem:weight}.

\smallskip

\textit{Case (I).} Find $\be_n =- \ga_n^2$ by Lemma~\ref{lem:findbg}. Then, the weight matrix $\mathcal N(\la_n)$ is uniquely determined by Lemma~\ref{lem:weight}.

\smallskip

\textit{Case (II).} By Lemma~\ref{lem:case2}, the functions $\Delta_{33}(\la) = -\Delta_{11}(\la)$ and $\Delta_{43}(\la)$ are uniquely specified. Therefore, by Lemma~\ref{lem:case2}, we find $\al_n$. Then, using \eqref{S4} and \eqref{D31}, determine 
$$
\Delta_{31}(\la_n) = -S_4(0, \la_n) = -\al_n \ga_n, \quad \Delta_{41}(\la_n) = -S_4'(0, \la_n) = -\al_n \xi_n. 
$$
Next, we find $m_{j1, \langle -1 \rangle}(\la_n) =-\dfrac{\Delta_{j1}(\la_n)}{\dot \Delta_{11}(\la_n)}$ for $j = 3, 4$. Consequently, the elements of the weight matrix $\mathcal N(\la_n)$ can be found from the relations \eqref{rel2} and \eqref{det0}.

\smallskip

\textit{Case (III).} Find $n_{21} = n_{43} = m_{43, \langle -1 \rangle}(\la_n)$ and $n_{41} = \xi_n^2$ by Lemma~\ref{lem:cases34}.

\smallskip

\textit{Case (IV).} Find the only non-zero element $n_{41} = \xi_n^2$ by Lemma~\ref{lem:cases34}.

\smallskip

\textit{Case (V).} For each $\mu_n$, find $n_{21}(\mu_n) = n_{43}(\mu_n) = m_{43,\langle -1\rangle}(\mu_n)$.
\end{proof}

\begin{remark}
Throughout this paper, we assume that, for each $n \ge 1$, there exists an eigenfunction $y_n(x)$ satisfying the condition $\int_0^1 y_n^2(x) \, dx = 1$. If the problem $\mathcal L$ is self-adjoint, then an eigenfunction can always be normalized by this condition. But in the non-self-adjoint case, we impose this assumption in order to avoid the situation $\int_0^1 y_n^2(x) \, dx = 0$. Anyway, in the case (I) of \eqref{cases}, the relation $\int_0^1 y_n^2(x) \, dx \ne 0$ is guaranteed even in the non-self-adjoint case (see Remark~\ref{rem:yn}).
\end{remark}

\begin{remark}
Note that the condition $\int_0^1 y_n^2(x) \, dx = 1$ defines the weight numbers $\ga_n$ and $\xi_n$ up to the sign: $\ga_n^{(1)} = -\ga_n^{(2)}$, $\xi_n^{(1)} = -\xi_n^{(2)}$. Anyway, according to the proof of Theorem~\ref{thm:uniq2}, the corresponding weight matrix $\mathcal N(\la_n)$ does not depend on the choice of either $\ga_n^{(1)}$, $\xi_n^{(1)}$ or $\ga_n^{(2)}$, $\xi_n^{(2)}$. The proof of Theorem~\ref{thm:uniq1} in Section~\ref{sec:uniq1} also does not depend on the choice of the sign.
\end{remark}

\section{Connection with Barcilon's problem} \label{sec:Bar}

In this section, we establish the connection between the inverse problems of McLaughlin \cite{McL82, McL86} and Barcilon \cite{Bar74-1, Bar74-2}. Barcilon's problem consists in the recovery of the fourth-order differential operator from three spectra. It is analogous to Borg's problem \cite{Borg46} of reconstructing the Sturm-Liouville potential from two spectra.
For convenience, we formulate Barcilon's problem (Inverse Problem~\ref{ip:Bar}) for the boundary conditions \eqref{bc}, while other types of separated boundary conditions also can be considered.

For $(j,k) = (1,2), (1,3), (2,3)$, denote by $\mathfrak S_{jk}$ the spectrum of the boundary value problem for equation \eqref{eqv} with the boundary conditions
\begin{equation} \label{bcBar}
U_j(y) = U_k(y) = 0, \quad V_1(y) = V_2(y) = 0.
\end{equation}

\begin{ip} \label{ip:Bar}
Given the three spectra $\mathfrak S_{12}$, $\mathfrak S_{13}$, $\mathfrak S_{23}$, recover the coefficients $p, q, a, b, c$ of the problem $\mathcal L$.
\end{ip}

In order to establish the connection between Inverse Problems~\ref{ip:main} and~\ref{ip:Bar}, we need several auxiliary lemmas.

\begin{lem} \label{lem:scf}
The spectra $\mathfrak S_{12}$, $\mathfrak S_{13}$, $\mathfrak S_{23}$ coincide with the sets of zeros of the functions $\Delta_{22}(\la)$, $\Delta_{32}(\la)$, $\Delta_{42}(\la)$, respectively.
\end{lem}

\begin{proof}
For $\mathfrak S_{12}$ and $\mathfrak S_{13}$, the assertion of the lemma readily follows from the comparison of the boundary conditions \eqref{bcjk} and \eqref{bcBar}.

Let us consider $\mathfrak S_{23}$ in more detail. In view of \eqref{initS}, $S_3(x, \la)$ and $S_4(x, \la)$ are two linearly independent solutions of \eqref{eqv} satisfying $V_1(y) = V_2(y) = 0$. Hence, the eigenvalues of $\mathfrak S_{23}$ coincide with the zeros of the characteristic function
$$
d(\la) := U_2(S_3) U_3(S_4) - U_3(S_3) U_2(S_4).
$$
Using \eqref{brack0}, \eqref{wron}, and \eqref{initS}, we derive
$$
0 = \langle S_3, S_4 \rangle_{|x = 1} = \langle S_3, S_4 \rangle_{|x = 0} = U_2(S_3) U_3(S_4) - U_1(S_3) U_4(S_4) + U_4(S_3) U_1(S_4) - U_3(S_3) U_2(S_4).
$$
This immediately implies
$$
d(\la) = U_1(S_3) U_4(S_4) - U_4(S_3) U_1(S_4),
$$
Therefore, $d(\la)$ is the characteristic function of the boundary value problem for equation \eqref{eqv} with the boundary conditions
$$
U_1(y) = U_4(y) = 0, \quad V_1(y) = V_2(y) = 0.
$$
The eigenvalues of the latter boundary value problem coincide with the zeros of $\Delta_{42}(\la)$, which completes the proof.
\end{proof}

\begin{lem} \label{lem:equivB}
Suppose that the eigenvalues $\{ \la_n \}_{n \ge 1}$ of the problem $\mathcal L$ are simple.
Then, the three spectra $\mathfrak S_{12}$, $\mathfrak S_{13}$, $\mathfrak S_{23}$ uniquely determine the data $\{ \la_n, \Delta_{32}(\la_n), \Delta_{42}(\la_n) \}_{n \ge 1}$ and vice versa.
\end{lem}

\begin{proof}
By definition, $\{ \la_n \}_{n \ge 1} = \mathfrak S_{12}$. Let the three spectra $\mathfrak S_{12}$, $\mathfrak S_{13}$, $\mathfrak S_{23}$ be given. It follows from Lemma~\ref{lem:scf} that they uniquely specify the functions $\Delta_{j2}(\la)$ for $j = 2, 3, 4$. Hence, the numbers $\Delta_{j2}(\la_n)$, $j = 3, 4$, $n \ge 1$ can be easily found.

Suppose that, on the contrary, for two problems $\mathcal L$ and $\tilde{\mathcal L}$, the following relations hold: 
\begin{equation} \label{eqsdaux}
\la_n = \tilde \la_n, \quad \Delta_{j2}(\la_n) = \tilde \Delta_{j2}(\la_n), \quad j = 3, 4, \quad n \ge 1.
\end{equation}
Let us show that $\mathfrak S_{k3} = \tilde{\mathfrak S}_{k3}$ for $k = 1, 2$.

Consider the functions $m_{j2}(\la) = -\dfrac{\Delta_{j2}(\la)}{\Delta_{22}(\la)}$, $j = 3, 4$. Obviously, they have simple poles $\{ \la_n \}_{n \ge 1}$ and $m_{j2,\langle -1 \rangle}(\la_n) = -\dfrac{\Delta_{j2}(\la_n)}{\dot \Delta_{22}(\la_n)}$. Hence, the equalities \eqref{eqsdaux} imply $m_{j2,\langle -1 \rangle}(\la_n) = \tilde m_{j2,\langle -1 \rangle}(\la_n)$. Therefore, the functions $(m_{j2} - \tilde m_{j2})$ are entire in $\la$ for $j = 3, 4$. It follows from \eqref{mjk} and Proposition~\ref{prop:estD} that
$$
|(m_{32} - \tilde m_{32})(\rho^4)| \le c_{\de} |\rho|^{-3}, \quad
|(m_{42} - \tilde m_{42})(\rho^4)| \le c_{\de} |\rho|^{-2}, \quad \rho \in G_{\de, 22\xi}, \quad |\rho| \to \infty.
$$
Consequently, by Liouville's Theorem, $m_{j2}(\la) \equiv \tilde m_{j2}(\la)$ for $j=3,4$. This implies $\Delta_{j2}(\la) \equiv \tilde \Delta_{j2}(\la)$, $j = 3, 4$, so the corresponding sets of zeros coincide: $\mathfrak S_{13} = \tilde{\mathfrak S}_{13}$, $\mathfrak S_{23} = \tilde{\mathfrak S}_{23}$.
\end{proof}

\begin{lem} \label{lem:relBM}
Suppose that $\Delta_{33}(\la_n) \ne 0$ for some $n \ge 1$. Then
\begin{equation} \label{relBM}
\Delta_{32}(\la_n) = \dot \Delta_{22}(\la_n) \ga_n^2, \quad
\Delta_{42}(\la_n) = \dot \Delta_{22}(\la_n) \xi_n \ga_n.
\end{equation}
\end{lem}

\begin{proof}
First, recall that 
$$
\be_n = m_{32,\langle -1 \rangle}(\la_n) = -\frac{\Delta_{32}(\la_n)}{\dot \Delta_{22}(\la_n)}.
$$
By Lemma~\ref{lem:findbg}, $\be_n = -\ga_n^2$. These two relations imply $\Delta_{32}(\la_n) = \dot \Delta_{22}(\la_n) \ga_n^2$.

Next, consider the relation \eqref{relm31} together with $m_{21}(\la) \equiv m_{43}(\la)$. Taking the residue at $\la = \la_n$, we get
$$
m_{42,\langle -1 \rangle}(\la_n) - m_{43}(\la_n) m_{32,\langle -1 \rangle}(\la_n) = 0.
$$
Using \eqref{mjk} and the relation $m_{43}(\la_n) = \frac{\xi_n}{\ga_n}$, we conclude that $\Delta_{42}(\la_n) = \dot \Delta_{22}(\la_n) \xi_n \ga_n$.
\end{proof}

Lemmas~\ref{lem:scf} and \ref{lem:relBM} imply the following corollary.

\begin{cor} \label{cor:equivBM}
Under the separation condition $(\mathcal S)$, McLaughlin's data $\{ \la_n, \ga_n, \xi_n \}_{n \ge 1}$ uniquely determine the data $\{ \la_n, \Delta_{32}(\la_n), \Delta_{42}(\la_n) \}_{n \ge 1}$ by the formulas \eqref{relBM} and vice versa. Thus, under the condition $(\mathcal S)$, Inverse Problems~\ref{ip:main} and \ref{ip:Bar} are equivalent.
\end{cor}

Applying Theorem~\ref{thm:uniq1} and Corollary~\ref{cor:equivBM}, we arrive at the uniqueness result for Barcilon's problem.

\begin{thm} \label{thm:Bar}
Suppose that the spectrum $\mathfrak S_{12}$ contains only simple eigenvalues and the separation condition $(\mathcal S)$ holds. Then, the three spectra $\mathfrak S_{12}$, $\mathfrak S_{13}$, $\mathfrak S_{23}$ uniquely specify the coefficients $p, q, a, b, c$.
\end{thm}

\begin{remark}
A uniqueness theorem for Barcilon's problem for the first time was formulated in \cite{Bar74-1}. But the proof in \cite{Bar74-1} is wrong. Specifically, in Section~4 of \cite{Bar74-1}, Barcilon considered the integrals $I_n$ and $J_n$ over the circular contours $\Gamma_n$ and got the relation $\lim\limits_{n \to \infty} I_n = \lim\limits_{n \to \infty} J_n$. But the rigorous analysis of the asymptotic behavior of integrands in the sectors $\Gamma_{\xi}$ defined in \eqref{defGa} shows that those limits do not exist. 
A correct proof of the uniqueness has been recently obtained by Guan et al \cite{GYB23} for distribution coefficients $p \in W_2^{-1}[0,1]$, $q \in W_2^{-2}[0,1]$. However, the results of \cite{GYB23} are limited to the case of simple poles of the Weyl-Yurko matrix. In Theorem~\ref{thm:Bar}, there is no requirement of the simplicity for the zeros of $\Delta_{11}(\la)$ and $\Delta_{33}(\la)$. The requirement of simplicity for the eigenvalues $\{ \la_n \}_{n \ge 1}$ also can be removed. From this viewpoint, Theorem~\ref{thm:Bar} is more general than the results of \cite{GYB23}.
\end{remark}

\begin{remark}
The separation condition in the uniqueness theorem for Barcilon's problem is crucial. If $(\mathcal S)$ is violated, then the spectral data $\mathfrak S_{12}$, $\mathfrak S_{13}$, $\mathfrak S_{23}$ of Barcilon provide less information than the spectral data $\{ \la_n, \ga_n, \xi_n \}_{n \ge 1}$ of McLaughlin. Indeed, as an example, consider the case (IV) of \eqref{cases}. Due to Lemma~\ref{lem:equivB}, there is a one-to-one correspondence between Barcilon's three spectra and the data $\{ \la_n, \Delta_{32}(\la_n), \Delta_{42}(\la_n) \}_{n \ge 1}$. But, in the case (IV), in view of Lemma~\ref{lem:weight}, we have $\Delta_{32}(\la_n) = \Delta_{42}(\la_n) = 0$. Thus, in Barcilon's problem, we have no additional information for the eigenvalue $\la_n$. However, in McLaughlin's problem, there is the additional information $\xi_n$, which allows us to find the element $n_{41}$ of the weight matrix $\mathcal N(\la_n)$ by Lemma~\ref{lem:cases34}. In Barcilon's problem, we cannot find this element.
\end{remark}

\medskip

{\bf Funding.} This work was supported by Grant 21-71-10001 of the Russian Science Foundation, https://rscf.ru/en/project/21-71-10001/.

\noindent Natalia Pavlovna Bondarenko \\

\noindent 1. Department of Mechanics and Mathematics, Saratov State University, \\
Astrakhanskaya 83, Saratov 410012, Russia, \\

\noindent 2. Department of Applied Mathematics and Physics, Samara National Research University, \\
Moskovskoye Shosse 34, Samara 443086, Russia, \\

\noindent 3. S.M. Nikolskii Mathematical Institute, Peoples' Friendship University of Russia (RUDN University), 6 Miklukho-Maklaya Street, Moscow, 117198, Russia, \\

\noindent e-mail: {\it bondarenkonp@info.sgu.ru}

\end{document}